\documentclass[12pt]{amsart}
\usepackage[margin=1in]{geometry}
\usepackage{amsmath,latexsym,amssymb,amsthm,graphicx}
\usepackage{mathtools}
\usepackage{subfig}

\usepackage{xfrac}
\usepackage{fix-cm}
\usepackage{color,transparent}
\usepackage{verbatim}

\usepackage[bookmarks=true,%
    colorlinks=true,%
    linkcolor=blue,%
    citecolor=blue,%
    filecolor=blue,%
    menucolor=blue,%
    urlcolor=blue]{hyperref}
\usepackage{cleveref}    
\usepackage{doi}
\usepackage{stmaryrd}
\usepackage{multirow}
\usepackage{float}
\usepackage{array}
\usepackage[shortlabels]{enumitem}
\usepackage{pdfsync}
\usepackage{tikz}
\usepackage{tikz-cd}
\usepackage{float}
\usepackage[T1]{fontenc}
\usepackage[utf8]{inputenc}
\usepackage[textsize=small]{todonotes}

\newcommand{\DMGtodo}[2][]{{}}

\usepackage{hyphenat}
\usepackage{booktabs}

\graphicspath{{mpdraws/}{draws/}}

\setenumerate[1]{label=(\arabic*)}

\numberwithin{equation}{section}
\numberwithin{figure}{section}
\newtheorem{theorem}[equation]{Theorem}

\newtheorem{mainthm}{Theorem}

\newtheorem{maincor}[mainthm]{Corollary}

\newtheorem{lemma}[equation]{Lemma}

\newtheorem{corollary}[equation]{Corollary}
\newtheorem{proposition}[equation]{Proposition}

\newtheorem*{ex*}{Exercise}

\theoremstyle{definition}
\newtheorem{remark}[equation]{Remark}

\newtheorem{definition}[equation]{Definition}
\newtheorem{question}[equation]{Question}

\newtheorem{notation}[equation]{Notation}

\theoremstyle{plain}

\newcommand\C{\mathcal{C}}

\DeclareMathOperator{\Teich}{Teich}

\DeclareMathOperator{\CAT}{CAT}

\DeclareMathOperator{\PC}{\mathbb{P}\mathcal{C}}
\DeclareMathOperator{\PCf}{\mathbb{P}\mathcal{C}_{\mbox{f}}}
\DeclareMathOperator{\Cf}{\mathcal{C}_{f}}

\DeclareMathOperator{\DNonPos}{DNonPos}

\newcommand{\ora}[1]{\overrightarrow{#1}}

\newcommand*{\wt}[1]{\widetilde{#1}}

\renewcommand{\epsilon}{\varepsilon}

\DeclareMathOperator{\PSL}{\mathit{PSL}}

\DeclareMathOperator{\supp}{supp}

\newcommand{\cal}[1]{\mathcal{#1}}
\newcommand{\bb}[1]{\mathbb{#1}}

\usetikzlibrary{arrows}
\tikzset{
    labl/.style={anchor=south, rotate=90, inner sep=.5mm}
}
\tikzstyle{every picture}=[> = to]
\tikzset{cdlabel/.style={execute at begin node=$\scriptstyle,execute at end node=$}}
\tikzset{implication/.style={double equal sign distance, -implies}}
\tikzset{biimplication/.style={double equal sign distance, implies-implies}}

\begin{document}

\title{Geodesic currents of coarse negative curvature}

\author[Jyothis]{Meenakshy Jyothis}
\address{Department of Mathematics\\
         University of Oklahoma\\
         601 Elm Ave Room 423, Norman, OK 73019\\
         USA}
\email{mjyothis@ou.edu}

\author[Martínez-Granado]{Dídac Martínez-Granado}
\address{Department of Mathematics\\University of Luxembourg\\Av. de la Fonte 6, Esch-sur-Alzette, L-4364, Luxembourg}       
\email{didac.martinezgranado@uni.lu}

\begin{abstract}
Strong hyperbolicity is a coarse notion of negative curvature, stronger than Gromov hyperbolicity, that includes all $\CAT(-k)$ metrics for $k>0$ and allows the use of dynamical techniques available in negative curvature, such as thermodynamical formalism. We prove that the subset of geodesic currents whose dual pseudometric is strongly hyperbolic is dense in the space of geodesic currents. The proof combines an elementary finite-cover argument with a characterization of strong hyperbolicity in terms of boundary data for pseudometrics dual to geodesic currents.

In contrast, we show that currents arising from non-positively curved metrics on the surface are not dense. As a consequence, we construct infinitely many pairwise non-roughly-isometric invariant strongly hyperbolic geodesic metrics on the universal cover of the surface which are not $\CAT(0)$.

Finally, we establish correlation counting results for the associated length spectra.
\end{abstract}

\maketitle

\section{Introduction}
Geodesic currents, introduced by Bonahon, provide a unifying framework for studying surface geometry, dynamics, and topology. 
They simultaneously generalize closed geodesics, measured laminations~\cite{Bonahon86:EndsHyperbolicManifolds,Bonahon88:GeodesicCurrent}, and a wide range of geometric structures on surfaces, including hyperbolic and non-positively curved Riemannian metrics~\cite{HP97:RigidityNegCurvedCone,CFF92:RigidityNonPosCurvedRiem,BL15:RigidityNonPosCurvedCone,Con18:MarkedNonpos}.

Throughout this paper, we assume that $S$ is a closed orientable surface of genus at least two, and let X denote $S$ equipped with a hyperbolic structure. Every geodesic current $\mu$ on a surface $X$ determines a dual invariant (pseudo-)metric $d_\mu$ on the universal cover $\widetilde{X}$, introduced in~\cite{BIPP21:Currents} and further developed in~\cite{DRMG23:Duals} (see also~\cite{COR22:Manhattan}). These pseudometrics are always Gromov hyperbolic, and their marked length spectra can be recovered via intersection numbers. In fact, when the currents are filling, they induce points in $\mathcal{D}(\Gamma)$, the space of classes of Gromov hyperbolic pseudometrics on $\pi_1(S)$ quasi-isometric to a word metric (\cite{Fur02:Coarse,OR22:SpaceMetric,DRMG23:Duals,CT25:Manhattan,CMGR26:GreenMetrics}). This gives a natural embedding from projective filling currents $\PCf(S)$ into $\mathcal{D}(\Gamma)$ (Proposition~\ref{prop:embedding}, see also~\cite[Proposition~6.12]{COR22:Manhattan}).

On the one hand, we show that geodesic currents associated to metrics of sharp negative or non-positive curvature---such as geodesic currents arising from non-positively curved Riemannian or locally $\CAT(0)$ metrics on $S$---have strong algebraic constraints in their marked length spectrum and thus fail to be dense in the space of geodesic currents.

On the other hand, we identify a dense subset of geodesic currents whose associated pseudometrics exhibit a coarse form of negative curvature: the class of \emph{strongly hyperbolic} pseudometrics. Introduced in~\cite{NicaSpakula2016}, strong hyperbolicity is a coarse analogue of negative curvature in the following sense:
\begin{enumerate}
\item it implies Gromov hyperbolicity~\cite[Theorem~4.2]{NicaSpakula2016};
\item negatively curved Riemannian metrics, as well as $\CAT(-k)$ metrics for $k>0$, are strongly hyperbolic~\cite[Theorem~5.1]{NicaSpakula2016};
\item strongly hyperbolic metrics induce H\"older potentials on $T^1X$~\cite[Example~4.12]{CMGR26:GreenMetrics}, allowing the use of sharp dynamical tools (e.g.\ thermodynamic formalism), much as in the setting of negatively curved Riemannian metrics.
\end{enumerate}

We exploit the third property above to derive counting results for pairs of strongly hyperbolic currents by applying classical dynamical counting theorems (Theorem~\ref{mainthm:counting}).

Finally, combining both the density of the strongly hyperbolic family of currents and the non-density of $\CAT(0)$ currents, we construct a dense subset $\mathcal{F}$ of geodesic currents whose associated pseudometrics $d_\mu$ are, in fact, $\pi_1(X)$-invariant \emph{geodesic metrics} on $\widetilde{X}$ that are strongly hyperbolic but \emph{not} $\CAT(0)$ (Corollary~\ref{cor:sh_not_cat0}).
 
\subsection{Density of strongly hyperbolic geodesic currents}

Suppose $X$ is a closed hyperbolic surface and
let $\pi \colon Y \to X$ be a finite-degree cover of $X$. We define the \emph{transfer map} associated to $\pi$ and denoted by $\Pi_Y: \mathcal{C}(Y) \rightarrow \mathcal{C}(X)$ as 
\[
\Pi_Y(\mu) = \frac{1}{n} \sum_{i=1}^n (g_i)_* \mu
\]
where $\mu \in \mathcal{C}(Y)$, and $g_i$ is a transversal of $\pi_1(Y) \leq \pi_1(X)$.
We define the following subspace of geodesic currents
\begin{equation}
\mathbb{R}\Teich^\infty(S)
\coloneqq
\left\{
\Pi_{Y}\!\left(\sum_{i=1}^k \lambda_i \mathcal{L}_{Y_i}\right)
:\;
k \in \mathbb{N},\;
\lambda_i \ge 0,\;
Y_i \in \Teich(Y),
Y \to X
\text{ finite cover}
\right\}.
\label{eq:rteichinfty}
\end{equation}
i.e. the subspace of the image under $\Pi_Y$ of non-negative linear combinations of hyperbolic Liouville currents on all finite covers of $S$. 
We prove the following density result.

\begin{mainthm}
The subspace of geodesic currents $\mathbb{R}\Teich^{\infty}(X)$ is dense in $\C(X)$.
\label{mainthm:rteich_dense}
\end{mainthm}

The relevance of $\mathbb{R}^{\infty}\Teich(X)$ is that it consists of geodesic currents whose dual pseudometrics are strongly hyperbolic metrics.
 
Given $\epsilon>0$, a metric space $X$ is said $\epsilon$-\emph{strongly hyperbolic} if for every points $x,y,z,w \in X$, we have
\begin{equation}
e^{\frac{\epsilon}{2}(d(x,z) + d(y,w))} \leq e^{\frac{\epsilon}{2}(d(x,y) + d(z,w))} + e^{\frac{\epsilon}{2}(d(x,w) + d(y,z))}.
\label{eq:intro_sh}
\end{equation}

We say that a geodesic current $\mu$ is \emph{strongly hyperbolic} if its dual pseudometric $d_\mu$ on $\widetilde{X}$ is strongly hyperbolic. 
While the metric $d_\mu$ depends on the choice of universal cover, its marked length spectrum does not (see Section~\ref{subsec:dual_current}), and strong hyperbolicity is in fact an invariant of this spectrum (Corollary~\ref{cor:stronghyp_roughsimilar}).

\begin{mainthm}
$\mathbb{R}\Teich^\infty(S)$ consists of strongly hyperbolic geodesic currents.
\label{mainthm:rinftyeich_sh}
\end{mainthm}

See Theorem~\ref{thm:sh_dense} for the proofs \Cref{mainthm:rteich_dense} and \Cref{mainthm:rinftyeich_sh}. As an immediate corollary of both, we have the following.

\begin{maincor}
The subspace of strongly hyperbolic currents is dense in the space of geodesic currents of $S$.
\label{maincor:density}
\end{maincor}

This result can be compared of recent work of Cantrell, Reyes, and the second author~\cite[Theorem~C]{CMGR26:GreenMetrics}, which shows that Green metrics associated to admissible random walks---also strongly hyperbolic---are dense in the space $\mathcal{D}(\Gamma)$ for any non-elementary hyperbolic group. 
By constrast, our Theorem~\ref{mainthm:rteich_dense} provides a family of metrics approximating any geodesic current entirely \emph{within} the space of geodesic currents. Our result does not follow from~\cite[Theorem~C]{CMGR26:GreenMetrics}, since it is not known whether such Green metrics arise from geodesic currentse~\cite[Question~11.8]{MGT25:Intersections} (and, indeed, this is likely false). 

The proof of Theorem~\ref{mainthm:rteich_dense} proceeds by an explicit geometric construction using finite covers. 
Given a closed curve $\gamma$ on $X$, Scott's LERF theorem for surface subgroups~\cite{Scott78:LERF} provides a finite cover in which $\gamma$ lifts to a union of simple closed geodesics. 
We approximate each simple geodesic by weighted Liouville currents, show that such combinations satisfy strong hyperbolicity, and then push forward to the original surface $X$ via the transfer map on currents. The above result uses a characterization of strong hyperbolicity in the setting of geodesic currents (see Theorem~\ref{thm:stronghyp_equivalent}). 

Observe that, by Equation~\ref{eq:intro_sh}, if $d$ is $\epsilon$-strongly hyperbolic, then $a \cdot d$ is $\epsilon/a$-strongly hyperbolic for every $a>0$. We will typically work with the representative in the projective class normalized to have entropy 1, i.e. $\hat{d} \coloneqq h(d) d$.

One consequence of Theorem~\ref{thm:stronghyp_equivalent} is that strong hyperbolicity for geodesic currents can be detected from boundary data. In particular, given a box $B= I \times J \subset G(S)$ of bi-infinite geodesics of $\wt{X}$ (for some bounded intervals $I, J$ of $S^1$), let $B^\perp = I^c \times J^c$ denote the box of geodesics determined by the complementary intervals defining $B$.
Theorem~\ref{thm:stronghyp_equivalent} implies, in particular, that among geodesic currents with entropy 1 (Eq.~\eqref{eq:entropy_mu}), the $1$-strongly hyperbolic ones are characterized by the inequalities
\[
1 \leq e^{-\mu(B)}+e^{-\mu(B^\perp)}
\]
for every box of geodesics $B \subset G(S)$
In fact, the Teichm\"uller space is exactly the locus where all these inequalities become equalities, by~\cite[Theorem~13]{Bonahon88:GeodesicCurrent}.

A qualitative version of strong hyperbolicity, still sharper than Gromov hyerbolicity, is that of \emph{strong bolicity} (see Section~\ref{sec:bolic}). It has played a key technical role in relevant rigidity conjectures relating geometric group theory with harmonic analysis and operator algebras (\cite{MineyevYu2002BaumConnes}).

We show that filling geodesic currents whose pseudometric is non-atomic are exactly the strongly bolic ones.
\begin{mainthm}
Let $\mu$ be a filling geodesic current.
Then $\mu$ is non-atomic if and only if $d_\mu$ is strongly bolic.
\label{mainthm:bolic}
\end{mainthm}
See Corollary~\ref{cor:fillingsb_noatoms} for the proof of the above statement. In fact, we prove a slightly stronger result (Theorem~\ref{thm:bolic}).

\subsection{$\CAT(0)$-geodesic currents are not dense and examples of strongly hyperbolic metrics that are not $\CAT(0)$}

We next show that the $\CAT(0)$ condition imposes strong algebraic constraints on the marked length spectrum, preventing density, and allowing us to construct examples of strongly hyperbolic metrics which are not $\CAT(0)$.

A metric space $X$ is called \emph{Ptolemaic} if for all $x,y,z,w \in X$,
\[
d(x,z)d(y,w) \leq d(x,y)d(z,w) + d(x,w)d(y,z).
\]
Examples include lifts of negatively curved Riemannian or locally $\CAT(0)$ Riemannian metrics on $S$.

We show that any $\pi_1(S)$-invariant Ptolemaic length metric on $\widetilde{S}$ satisfies a countable family of homogeneous quadratic inequalities on its marked length spectrum:

\begin{mainthm}
Let $\rho$ be a Ptolemaic metric on $S$ which is also a length metric. 
For any $a,b \in \pi_1(S)$ with crossing axes, and every $n \in \mathbb{N}$,
\[
\ell_{\rho}(a^n)\,\ell_{\rho}(b^n) 
\leq 
\frac{\ell_{\rho}(a^n b^n)^2}{4}
+
\frac{\ell_{\rho}(a^n b^{-n})^2}{4}.
\]
\label{mainthm:Ptolemy_lengths}
\end{mainthm}

See Theorems~\ref{thm:intersections} and~\ref{thm:Ptolemy_lengths} below.
These inequalities provide a strong rigidity constraint. 
We then show that intersection functionals associated to non-simple curves violate these relations, and deduce:

\begin{maincor}
The Liouville currents induced by Ptolemaic metrics on $S$ (in particular, locally $\CAT(0)$ or negatively curved Riemannian metrics) are not dense in the space of geodesic currents.
\label{maincor:negnotdense}
\end{maincor}

See Corollary~\ref{cor:neg_notdense} below. This contrasts with~\cite[Lemma~2.8]{Ham99:Cocycles}, where density is asserted for currents arising from negatively curved metrics. Corollary~\ref{maincor:negnotdense} shows that the statement is not correct as stated and that the proof contains a gap. See the discussion following Corollary~\ref{cor:neg_notdense} for further details. Nevertheless, the statement of~\cite[Lemma~2.8]{Ham99:Cocycles} was one of the inspirations for Theorem~\ref{mainthm:rteich_dense} and Corollary~\ref{maincor:density}.

Finally, our methods allow us to construct infinitely many \emph{distinct rough isometry classes} of strongly hyperbolic geodesic metrics which are not $\CAT(0)$.

\begin{maincor}
There exists a subset $\mathcal F \subset \mathbb{R}\Teich^\infty(S)$ which is dense in the space of geodesic currents such that, for every $\mu\in\mathcal F$, the associated $\pi_1(S)$-invariant pseudometric $d_\mu$ on $\wt S$ is a strongly hyperbolic geodesic metric but not Ptolemaic, and hence not $\CAT(0)$.
\label{maincor:sh_not_cat0}
\end{maincor}

See Corollary~\ref{cor:sh_not_cat0} below.

\subsection{A correlation counting result}

The same techniques as in~\cite{Sharp98:Hyperbolic} and~\cite{DM23:Correlation},
ultimately relying on a theorem of Lalley~\cite{Lalley97:AxiomA},
yield a correlation counting theorem for strongly hyperbolic currents.

Given two strongly hyperbolic currents $\mu_1,\mu_2$, their \emph{Manhattan curve}
$\mathcal C(\mu_1,\mu_2)$  (Section~\ref{subsubsec:manhattan}) is the boundary of the set
\[
\left\{
(a,b)\in\mathbb R^2:
\sum_{[\gamma]\in[\Gamma]}
e^{-a\ell_{\mu_1}([\gamma])-b\ell_{\mu_2}([\gamma])}
<
\infty
\right\}.
\]
The \emph{renormalized length} spectrum of a geodesic current $\mu$ is
\[
L_\mu\coloneqq h(\mu)\ell_\mu,
\]
where $h(\mu)$ denotes the topological \emph{entropy} of the geodesic current, i.e., the exponential growth rate of curves with respect to $\ell_\mu$ (Definition~\ref{eq:entropy_mu}).
For any pair of strongly hyperbolic geodesic currents $\mu_1,\mu_2$, the Manhattan curve $\mathcal C(\mu_1,\mu_2)$ is real analytic (\cite[Theorem~4.16]{CT25:Manhattan} and paragraph after its proof) and, whenever $\mu_1,\mu_2$ have distinct renormalized length spectra, its corresponding abscissa function is a strictly convex curve (\cite[Proposition~2.3]{CR25:Marked}).

\begin{mainthm}
\label{mainthm:counting}
Fix $\varepsilon>0$. Let $\mu_1,\mu_2$ be strongly hyperbolic currents with distinct renormalized length spectra. Then there exist constants
\[
C=C(\varepsilon,\mu_1,\mu_2)>0
\qquad\text{and}\qquad
M=M(\mu_1,\mu_2)\in(0,1)
\]
such that
\begin{equation}
\#\Big\{
[\gamma]\in[\Gamma]:
L_{\mu_1}([\gamma])
\in
(x,x+h(\mu_1)\varepsilon),
\
L_{\mu_2}([\gamma])
\in
(x,x+h(\mu_2)\varepsilon)
\Big\}
\sim
C\frac{e^{Mx}}{x^{3/2}}.
\label{eq:correlation}
\end{equation}

Moreover,
\begin{equation}
\label{eq:correlation_formula}
M(\mu_1,\mu_2)
=
\frac{a}{h(\mu_1)}
+
\frac{b}{h(\mu_2)},
\end{equation}
where $(a,b)\in\mathcal C(\mu_1,\mu_2)$ is the unique point at which the tangent line is parallel to the line joining $(h(\mu_1),0)$ and $(0,h(\mu_2))$.
\end{mainthm}

The counting statement is proved in Theorem~\ref{thm:main}, while the characterization of $M(\mu_1,\mu_2)$ in terms of the Manhattan curve is established in Theorem~\ref{thm:Manhattan}.

As an immediate consequence of Theorems~\ref{mainthm:counting},
\ref{mainthm:rinftyeich_sh}, and~\ref{mainthm:rteich_dense}, we obtain the following corollary.

\begin{maincor}
The correlation counting result of Theorem~\ref{mainthm:counting}
holds for pairs of geodesic currents belonging to a dense subset of the space of geodesic currents.
\end{maincor}

\subsection{Organization}

Section~\ref{sec:background} reviews the necessary background on Gromov hyperbolicity, strong hyperbolicity, geodesic currents, pseudometrics, and cocycles. Section~\ref{sec:stronghyp} proves the characterization theorem for strong hyperbolicity of pseudometrics dual to geodesic currents (Theorem~\ref{thm:stronghyp_equivalent}). Section~\ref{sec:bolic} characterizes which pseudometrics dual to geodesic currents are strongly bolic (Theorem~\ref{mainthm:bolic}), and shows that, for filling currents, non-atomicity is equivalent to strong bolicity. These results are subsequently used in Section~\ref{sec:stronghyp}. Section~\ref{sec:neg_notdense} establishes the length spectrum relations satisfied by currents dual to $\CAT(0)$ metrics (Theorem~\ref{mainthm:Ptolemy_lengths}), proves that such currents are not dense (Corollary~\ref{maincor:negnotdense}), and constructs infinitely many invariant strongly hyperbolic geodesic metrics on $\wt{X}$ which are not $\CAT(0)$ (Corollary~\ref{maincor:sh_not_cat0}). Section~\ref{sec:stronglyhyp} proves the density results for strongly hyperbolic metrics, establishing Theorems~\ref{mainthm:rinftyeich_sh},~\ref{mainthm:rteich_dense}, and Corollary~\ref{maincor:density}. Section~\ref{sec:correlation} proves the correlation counting theorem (Theorem~\ref{mainthm:counting}). 

The appendix is divided into two parts. Section~\ref{sec:signed} proves injectivity of the marked length spectrum for signed geodesic currents and derives a sufficient criterion for a signed current to be a weighted signed multicurve. Section~\ref{sec:elliptic-modulus} establishes estimates for the function comparing hyperbolic Liouville currents associated to distinct points of Teichm\"uller space.

\subsection{Acknowledgments}
\label{sec:acknowledgments}
D. Mart\'inez-Granado thanks Stephen Cantrell, Giuseppe Martone, Bea Pozzetti, and Eduardo Reyes, Ser Peow Tan, Dylan Thurston and Tengren Zhang for helpful conversations. He specially thanks Eduardo Reyes for pointing out to the construction of strongly hyperbolic metrics which are not $\CAT(0)$. M. Jyothis would like to thank Giuseppe Martone and Jenya Sapir for helpful discussions. Support for D. Mart\'inez-Granado was provided by the Luxembourg National Research Fund (AFR/Bilateral-ReSurface 22/17145118), the National University of Singapore, and the Marie Skłodowska-Curie Action CurrGeo grant (101154865). He also thanks the Department of Mathematics at the National University of Singapore for their generous hospitality and for providing an excellent working environment during the period in which most of this work was carried out.

\section{Background}
\label{sec:background}

\subsection{Gromov hyperbolicity}

We say $d \colon X \times X \to \mathbb{R}$ is a \emph{pseudometric} on $X$ if, for every $x,y,z \in X$, it satisfies $d(x,x)=0$, $d(x,y) \geq 0$, $d(x,y)=d(y,x)$ and the \emph{triangular inequality} $d(x,z) \leq d(x,y) + d(y,z)$, but does not necessarily satisfy the condition that $d(x,y)=0$ implies $x=y$.
We say that a pseudometric $d$ on $X$ is \emph{$\delta$-Gromov hyperbolic} if there exist $\delta>0$ so that for every four tuple of points $x,y,z,w \in X$,
\begin{equation}
d(x,z)+d(y,w) \leq \max \{ d(x,y)+d(z,w), d(x,w) + d(y,z) \} + 2\delta.
\label{eq:gromov_hyp}
\end{equation}
This condition can be equivalently rewritten as
\[
(x|y)_w \geq \min \{ (x|z)_w, (z|y)_w \}-\delta,
\]
where $(x|y)_w$ denotes the \emph{Gromov product} of $x,y \in X$ based at $w \in X$, and is defined as
\[
(x|y)_w \coloneqq \frac{1}{2}\left( d(x,w)+d(y,w)-d(x,y)\right).
\]
Sometimes we will simply say $X$ is a \emph{hyperbolic space}.
Hyperbolic spaces can be compactified using their ideal \emph{Gromov boundary $\partial X$} consisting of equivalence classes of \emph{divergent sequences}. A sequence of elements $(x_n)_{n=0}^{\infty}$ diverges if $(x_n | x_m)_o$ for some $o \in X$, as $\min \{ n, m\}$ diverges to infinity. 
Two divergent sequences $(x_n)_{n=0}^{\infty}$ and $(y_n)_{n=0}^{\infty}$ are \emph{equivalent} if $(x_n | y_m)_o$ diverges as $\min \{n,m \}$, tends to infinity.

The Gromov product extends to points in $X \cup \partial X$, however this extension need not be continuous. The extension is defined as follows. Let

\[
(\xi | \eta)_o = \sup \liminf_{n \to \infty} (x_n | y_n)_o \colon \xi = [(x_n)], \eta = [(y_n)]\},
\]
for $\xi, \eta \in X \cup \partial X$, where if $\xi, \eta$ is in $X$, then we take $(x_n)$ as the constant sequence $x_n = \xi$ for all $n \geq 0$.

We say a (pseudo)-metric space $X$ is \emph{$\alpha$-roughly geodesic} if there exist a uniform constant $\alpha>0$ so that for every two points $x,y \in X$, there exist a finite sequence of points $(x_i)_{i=0,\cdots,n}$ in $X$ so that $x_0=x$ and $x_n=y$ satisfying
\[
i-j - \alpha \leq d(x_i,x_j) \leq i-j + \alpha
\]
for every $i \geq j$.
We say $X$ is \emph{roughly geodesic} if it is $\alpha$-roughly geodesic for some $\alpha$.

We say that a finitely generated group $\Gamma$ is \emph{Gromov hyperbolic} if some/any of its Cayley graphs with the edge-metric is a Gromov hyperbolic space.
We say a Gromov hyperbolic group is \emph{non-elementary} if $\partial \Gamma$ has at least 3 points (in which case it has infinitely many, in fact).

Given a pseudometric $d$ in $\Gamma$, we define the \emph{stable length} of $g \in \Gamma$, as follows.
\[
\ell_d(g) \coloneqq \lim_n \frac{d(o,g^n)}{n},
\]
where $o$ denotes the identity element of $\Gamma$. This is a conjugacy class invariant.

\subsection{Strongly hyperbolicity and bolicity}

Following Nica--\v{S}pakula~\cite{NicaSpakula2016}, we define the following.

\begin{definition}[$\epsilon$-strongly hyperbolic]
Given $\epsilon>0$, a metric space $X$ is said $\epsilon$-strongly hyperbolic if for every points $x,y,z,w \in X$, we have
\begin{equation}
e^{\frac{\epsilon}{2}(d(x,z) + d(y,w))} \leq e^{\frac{\epsilon}{2}(d(x,y) + d(z,w))} + e^{\frac{\epsilon}{2}(d(x,w) + d(y,z))}.
\label{eq:strong_hyp_epsilon}
\end{equation}
\label{def:epsilon_stronghyp}
\end{definition}
The notion of strong hyperbolicity implies Gromov hyperbolicity (\cite[Theorem~4.2]{NicaSpakula2016} and Proposition~\ref{prop:nica_spakula} below).

There is also the following equivalent definition, introduced by Nica--\v{S}pakula.

\begin{definition}[$(A_0,B_0,C_0)$-strongly hyperbolic]
Given positive constants $A_0,B_0$ and $C_0$, a metric space $X$ is said $(A_0,B_0,C_0)$-strongly hyperbolic if it is hyperbolic (according to Equation~\ref{eq:gromov_hyp}) and if for every points $x,y,z,w \in X$, we have that, if
\[
d(x,z) + d(y,w)-d(x,y)-d(z,w) > A \geq A_0,
\]
then
\[
d(x,z) + d(y,w)-d(x,w)-d(y,z) < B_0 e^{-C_0 \cdot A}.
\]
\end{definition}

In \cite[Lemma~6.2]{BS16:StrongHyp}, Nica--\v{S}pakula prove the following.

\begin{proposition}
\label{prop:nica_spakula}
\begin{enumerate}
    \item If $X$ is $\epsilon$-strongly hyperbolic, then:
    \begin{enumerate}
        \item $X$ is $\frac{\log(2)}{\epsilon}$-Gromov hyperbolic;
        \item $X$ is $(A_0(\epsilon),\, 4/\epsilon,\, \epsilon/2)$-strongly hyperbolic,
        where $A_0(\epsilon)$ is given implicitly.
    \end{enumerate}

    \item If $X$ is $(A_0,B_0,C_0)$-strongly hyperbolic with Gromov hyperbolicity constant $\delta$, then $X$ is $\epsilon$-strongly hyperbolic, where
    $\epsilon \coloneqq \min\left\{
        \frac{1}{2\delta},
        \frac{2\log(2)}{A_0}
    \right\}.$
\end{enumerate}
\end{proposition}

\begin{remark}
Observe that if $X$ is $\epsilon$-strongly hyperbolic, then it is also $\epsilon'$-strongly hyperbolic for every $\epsilon' \leq \epsilon$.
It thus makes sense to consider the supremum $\epsilon^*$ of the constants $\epsilon$ of $\epsilon$-strong hyperbolicity for a given metric space $X$. Note that, by the above, if $\delta^*$ is the infimum of the constants $\delta$ of $\delta$-hyperbolicity of $X$ (Equation~\ref{eq:gromov_hyp}), then
$\epsilon^* \leq \frac{1}{2\delta^*}$. 
\end{remark}

There is a related notion, weaker than that of strong hyperbolicity, namely, \emph{strong bolicity}, originally defined by Lafforgue~\cite{Lafforgue2002}. In the context of roughly geodesic spaces, strong bolicity can be formulated as follows~\cite[p.~960]{NicaSpakula2016}; this is the definition we will use.

\begin{definition}
We say a that a roughly geodesic space $X$ is \emph{strongly bolic} if $X$ satisfies 
\[
d(x,z) + d(y,w)-d(x,w)-d(y,z) \to 0
\]
as
\[
d(x,z) + d(y,w)-d(x,y)-d(z,w) \to \infty
\]
for every 4-tuple of points $x,y,z,w \in X$.
\label{def:strong_bolicity}
\end{definition}

From this definition, the following is straightforward (c.f. ~\cite[Theorem~1.1]{NicaSpakula2016}).

\begin{lemma} If $X$ is strongly hyperbolic and roughly geodesic, then $X$ is strongly bolic.
\label{lem:sh_implies_sb}
\end{lemma}

\subsection{Space of (pseudo-)metrics} 
\label{subsec:pseudo_metrics_space}
For the following, we consider the references~\cite{Fur02:Coarse,OR22:SpaceMetric,R23:Thesis}.
For $\Gamma$ a non-elementary hyperbolic group, let $D(\Gamma)$ denote the space of all Gromov hyperbolic, left-invariant pseudometrics on $\Gamma$ that are quasi-isometric (q.i) to a word-metric. We consider pseudometrics up to rough similarity: two pseudometrics $d_1$ and $d_2$ on $\Gamma$ are said to be equivalent if there exist constants $A,K>0$ such that
\[
|K \cdot d_1 - d_2|<A.
\]
When $K=1$, we say $d_1,d_2$ are \emph{roughly isometric}.
We define $\mathcal{D}(\Gamma) \coloneqq D(\Gamma)/\sim$, where $\sim$ denotes the equivalence relation of rough similarity.
For any pseudometric $d \in D(\Gamma)$, and $g,h \in \Gamma$,
one has $d(x,y)=d(y^{-1}x,o) \geq \ell_d(y^{-1}x)$.
Since $\ell_d(y^{-1}x) >0$ whenever $y^{-1}x$ is non-torsion, the 
subgroup $\{ x \in \Gamma \mid d(o,x)=0\}$ is torsion, and hence
finite. In particular, if $\Gamma$ is torsion-free (such as when
$\Gamma=\pi_1(S)$), then every pseudometric on $D_\Gamma$ is an actual metric~\cite[Remark~4.1.3]{R23:Thesis}. An important invariant of $d \in D(\Gamma)$ is the \emph{entropy} of $d$, the \emph{positive} real number given by
\begin{equation}
h(d) \coloneqq \lim_{n \to \infty} \frac{\log |\{g \in \Gamma : \ell_d(g)<n\}|}{n}.
\label{eq:entropy_d}
\end{equation}
Examples of elements in $\mathcal{D}(\Gamma)$ include equivalence classes of word-metrics on $\Gamma$ induced by finite, symmetric generating sets, as well as pullback of metrics arising from geometric (i.e., proper and cocompact) actions on Gromov hyperbolic, roughly geodesic spaces. We will see below that the pseudometrics associated to geodesic currents also provide examples of elements in $\mathcal{D}(\Gamma)$.

The space $\mathcal{D}(\Gamma)$ can be equipped with a natural symmetric metric, defined as follows.
\begin{equation}
\Delta([d_1],[d_2]) \coloneqq \log \left(\sup_{[g] \in [\Gamma']} \frac{\ell_{d_2}([g])}{\ell_{d_1}([g])} \sup_{[g] \in [\Gamma']} \frac{\ell_{d_1}([g])}{\ell_{d_2}([g])} \right),
\label{eq:delta_metric}
\end{equation}
where $\Gamma'$ denotes the set of loxodromic elements of $\Gamma$.

\subsection{Geodesic currents}\label{ssec:currents}

Let $S$ be a closed, connected, orientable hyperbolic surface. Fix a hyperbolic metric $X$ on $S$ once and for all. The universal cover $\wt{X}$ of $X$ is isometric to the Poincar\'e disk with the standard hyperbolic metric. Let $\mathcal{G}(\wt{X})$ denote the space of unparametrized bi-infinite geodesics of $\wt{X}$. It is known that $\mathcal{G}(\wt{X})$ is homeomorphic to $(S^1 \times S^1 - \Delta)/\mathbb{Z}_2$. A \emph{geodesic current} is a locally finite, positive, $\pi_1(X)$-invariant Borel  measure on $\mathcal{G}(\wt{X})$. We will let $\mathcal{C}(X)$ denote the space of all geodesic currents on $X$. Examples of such objects are geodesic currents $\delta_c$ arising from a close geodesic $c$ in $X$: precisely, $\delta_c$ is obtained as the Dirac measure on the full lift of the closed geodesic to $\wt{X}$. Furthermore, every choice of hyperbolic metric $Y$ on $S$ induces a geodesic current $\mathcal{L}_Y$, called the \emph{(hyperbolic) Liouville current} via the hyperbolic cross-ratio on the boundary of at infinity of $\wt{Y}$.
Precisely, for every point in Teichm\"uller space represented by a marked hyperbolic surface $\mathcal{Y}\coloneqq[(Y,f)]$, where $f \colon X \to Y$ is a quasi-conformal homeomorphism, we obtain an equivariant quasiconformal homeomorphism
\[
\wt{f} \colon \wt{X} \to \wt{Y}.
\]
There is a $\pi_1(Y)$-invariant measure $L_Y$ on $\mathcal{G}(\wt{Y})$ (the space of geodesics of $\wt{Y})$), given by the logarithm of the hyperbolic cross ratio on $\partial \wt{Y}$.
We can push $L_Y$ forward to $\mathcal{G}(\wt{X})$ via $f^{-1}$, obtaining:
\[
\mathcal{L}_{\mathcal{Y}} \coloneqq (f^{-1})_* L_Y.
\]
Bonahon showed the map $\mathcal{Y} \mapsto \mathcal{L}_\mathcal{Y}$ gives a proper embedding of $\mathcal{T}(X)$ into $\mathcal{C}(X)$. In fact, the embedding is into $\mathbb{P}\mathcal{C}(X) \coloneqq (\mathcal{C}(X)-\{ 0\}) / \mathbb{R}_{>0}$, the \emph{projectivized space of geodesic currents}. 
Many other metric structures induce geodesic currents. See, for example,~\cite{Otal90:SpectreMarqueNegative,CFF92:RigidityNonPosCurvedRiem,HP97:RigidityNegCurvedCone,BL17:FlatRigidity,Erlandsson:WordLength}.

The space $\cal C(X)$ of geodesic currents is a convex cone in an infinite dimensional vector space. Bonahon \cite{Bonahon86:EndsHyperbolicManifolds} extended the geometric intersection pairing on closed curves to the space of geodesic currents. That is, he showed that there exists a non-negative, symmetric, bilinear and continuous pairing
\[
i\colon \cal C(X)\times\cal C(X)\to \bb R_{\geq 0}
\]
such that $i(\delta_c,\delta_d)$ equals the geometric intersection number of the closed geodesics $c$ and $d$.

Given a geodesic current $\nu$, we can use the intersection number to define its length spectrum $\ell_\nu\colon[\Gamma]\to\bb R_{\geq 0}$ as $\ell_\nu([g])=i(\nu,\delta_{[g]})$. The \emph{systole} of $\nu$ is then $
\rm{sys}(\nu):=\inf_{[g]\in[\Gamma]}\ell_\nu([g])$.
 \cite[Corollary~1.5]{BIPP21:Currents} shows that $\rm{sys}\colon\cal C(S)\to[0,\infty)$ is a continuous function. 

A geodesic current $\nu$ is {\em filling} if for every $\mu \neq 0$, we have $i(\nu, \mu) >0$. We denote the space of filling geodesic currents by $\cal C_f(X)$, and its projectivization by $\mathbb{P}\cal C_f(X)$.
We define the  \emph{entropy} of a filling geodesic current $\nu$ by
\begin{equation}
h(\nu)\coloneqq \lim_{n\to \infty}\frac{1}{n}\log |\{[g]\in[\Gamma]\mid \ell_{\nu}([g])<n\}|.
\label{eq:entropy_mu}
\end{equation}
Examples of filling currents are hyperbolic Liouville currents, or currents associated to filling closed curves.
Examples of non-filling currents are \emph{measured laminations}, i.e. geodesic currents $\alpha$ satisfying $i(\alpha,\alpha)=0$~\cite[Proposition~14]{Bonahon88:GeodesicCurrent}.
We will denote the space of measured laminations by $\mathcal{ML}(S)$.

\subsection{Dual pseudometric of a geodesic current}
\label{subsec:dual_current}
Let $Y$ be an arbitrary hyperbolic structure on the surface $S$.
Any geodesic current induces a pseudometric $d_{\mu}$ on $\wt{Y}$ as follows.

\begin{definition}[pseudometric]\label{def:current_pseudo} Given $x,y \in \wt{Y}$, define, following~\cite[Section~4]{BIPP21:Currents}
\[
d_{\mu}^Y(x,y) \coloneqq \frac{1}{2}\mu(G[x,y)) + \frac{1}{2}\mu(G(x,y]),
\]
where, in general, the set $G(I)$, for $I$ a geodesic segment in $\wt{Y}$, is the set of bi-infinite geodesics of $\wt{Y}$ intersecting $I$ transversely. In particular, $G[x,y))$ is the set of bi-infinite geodesics intersecting the geodesic segment $[x,y)$ transversely.  
\end{definition}

This pseudometric on $\wt{Y}$ satisfies several known and important properties, which we list here as a proposition.

\begin{proposition}
\label{prop:dmu}
Given a hyperbolic structure $Y$ on $S$ and a geodesic current $\mu$ in $\mathcal{C}(Y)$, the following are true:
\begin{enumerate}[label=(\alph*),ref=\thetheorem~(\alph*)]
\item\label{item:dmu_straight}
$d_\mu^{Y}$ is a \emph{straight pseudometric}: for any triple of points
$(x,y,z) \in l^3$, where $l$ is a hyperbolic geodesic in $\wt{Y}$,
ordered $x<y<z$ along $l$, we have
\[
d_{\mu}^Y(x,y)+ d_{\mu}^Y(y,z)=d_{\mu}^Y(x,z).
\]

\item\label{item:dmu_length_intersection}
The stable length of $d_{\mu}^Y$ is equal to the \emph{intersection number}
of $\mu$:
\begin{equation}
\ell_{d_{\mu}^Y}([g])=i(\mu, [g]).
\end{equation}

\item\label{item:dmu_hyperbolic}
$d_{\mu}^{Y}$ is \emph{Gromov hyperbolic} in the sense of
Equation~\eqref{eq:gromov_hyp}.

\item\label{item:dmu_qi}
If $\mu$ is a filling geodesic current, then $d_{\mu}^{Y}$ is
\emph{equivariantly quasi-isometric} to the hyperbolic metric $\wt{X}$
 and hence also quasi-isometric to any word metric on
$\pi_1(X)$; see~\cite[Theorem~C]{DRMG23:Duals} and also
\cite[Theorem~1.11]{COR22:Manhattan}.

\item\label{item:dmu_roughlygeodesic}
$d_\mu^{Y}$ is \emph{roughly geodesic}.
\end{enumerate}
\end{proposition}
\begin{proof}
Item~\ref{item:dmu_straight} follows is proven in~\cite[Proposition~4.1]{BIPP21:Currents}.
Item~\ref{item:dmu_length_intersection} is \cite[Lemma~4.7]{BIPP21:Currents}.
Item~\ref{item:dmu_hyperbolic} is~\cite[Theorem~A]{DRMG23:Duals} (c.f.~\cite[Proposition~6.12]{CT25:Manhattan}).
Item~\ref{item:dmu_roughlygeodesic}~follows from~\cite[Lemma~3.2]{COR22:Manhattan} and~\cite[Lemma~6.12]{COR22:Manhattan}).
\end{proof}

Note that from \ref{item:dmu_length_intersection}, it follows that the definition of entropy of $d_\mu^{Y}$ (Def.~\ref{eq:entropy_d}) coincides with that of the entropy of the geodesic current $\mu$ (Def.~\ref{eq:entropy_mu}).

We can also see the class of a filling geodesic current $\mu$ as a point in $\mathcal{D}(\Gamma)$ somewhat canonically.  Given $\mathcal{Y} \coloneqq [(Y,f)] \in \Teich(S)$, where $f \colon X \to Y$ is a marking of $X$. Let $f_\# \colon \pi_1(X) \to \pi_1(Y)$ be the induced map at the level of fundamental groups, and $f_* \colon \C(X) \to \C(Y)$ be the induced homeomorphism at the level of geodesic currents.
This pseudometric naturally induces a class of metrics on $\Gamma=\pi_1(X)$.

\begin{definition}
For every $w \in \wt{Y}$, we can define a pseudometric on $\Gamma$ as follows
\[
(d_\mu^{\mathcal{Y}})_\Gamma(g,h) \coloneqq d_{f_*(\mu)}^Y(f_\#(g)w,f_\#(h)w).
\]
\label{def:pseudometric_G}
\end{definition}
From Proposition~\ref{item:dmu_length_intersection} it follows that
\[
\ell_{d_\mu^{\mathcal{Y}}}(g)=i(\mu, g)
\]
for every $\mathcal{Y} \in \Teich(S), g\in \pi_1(X)$.
In particular, $[(d_{\mu}^{Y_1})_\Gamma] \in \mathcal{D}(\Gamma)$, and the rough isometry class of $[(d_{\mu}^{Y_1})_\Gamma]$ and $[(d_{\mu}^{Y_2})_\Gamma]$ is the same for every $Y_1, Y_2 \in \Teich(S)$~\cite{DRMG23:Duals},\cite{COR22:Manhattan}. 

\begin{notation}
    Since the stable length of this metric does not depend on underlying $\mathcal{Y} \in \Teich(S)$, or the choice of basepoint $w \in \wt{Y}$, we will take $Y=X$ and the marking to be the identity.
In particular, we will sometimes omit the hyperbolic metric and write $d_\mu$ to refer to the pseudometric $d_\mu^X$ on $\wt{X}$.
\end{notation}

The following was first implicitly observed in~\cite{Sap23:Metric}, but, for completeness, we provide an explicit proof here (for related results see~\cite[Proposition~6.12]{CT25:Manhattan} and~\cite[Theorem~F]{DRMG23:Duals}).

\begin{proposition}
Let $\mathbb{P}\mathcal{C}_f(S)$ denote the space of projective filling currents equipped with the weak$^*$-topology, and let $\mathcal{D}(\Gamma)$ be equipped with the topology induced by the metric $\Delta$.
Then the map $\mathbb{P}\mathcal{C}_{f}(S) \to \mathcal{D}(\Gamma), [\mu] \mapsto [d_{\mu}]$ is an embedding.
\label{prop:embedding}
\end{proposition}
\begin{proof}
The injectivity follows by a result of Otal~\cite[Th\'eor\`eme~1]{Otal90:SpectreMarqueNegative}. The continuity of the map and its inverse follows from the equivalence of the weak$^*$-topology and the subspace topology induced by the distance $\Delta$ on $\mathbb{P}\mathcal{C}_f(S)$ (Equation~\eqref{eq:delta_metric}). This in turn follows from the fact that $\Delta$ is continuous with respect to the weak$^*$ topology. Indeed, the suprema involved in the definition of $\Delta$ can be written as maxima of continuous functions indexed over the \emph{compact} space of projective geodesic currents~\cite[Corollary~5]{Bonahon88:GeodesicCurrent}. Indeed the function of type $\Cf(S) \times \Cf(S) \times \PC(S) \to \mathbb{R}$ given by
\[
(\mu_1,\mu_2,[\mu]) \mapsto \frac{\ell_{d_{\mu_2}}(\mu)}{\ell_{d_{\mu_1}}(\mu)}
\]
is continuous.
This follows by Proposition~\ref{item:dmu_length_intersection}, from the fact that $i(\cdot, \cdot)$ is a continuous, bilinear function on the space of geodesic currents~\cite[Proposition~4.5]{Bonahon86:EndsHyperbolicManifolds}, and $i(\mu_i,\cdot)>0$ since $\mu_i$ are filling for $i=1,2$.
\end{proof}

\subsection{From strongly hyperbolic to H\"older cocycles}
\label{subsec:Holderocycle}

Let $d$ be a strongly hyperbolic metric on a non-elementary Gromov hyperbolic group $\Gamma$ and fix $w\in\Gamma$.
The associated \emph{Busemann cocycle} is
\[
\beta_w(x,\xi)
=
\lim_{n\to\infty}d(x,\xi_n)-d(w,\xi_n),
\]
where $\xi=[(\xi_n)]\in\partial\Gamma$

The following properties are known.

\begin{lemma}\label{lem:strong-cocycle-holder}
If $d$ is strongly hyperbolic, then:
\begin{enumerate}
\item the limit defining $\beta_w$ exists and satisfies
\[
\beta_w(x,\xi)=d(w,x)-2(x|\xi)_w;
\]
and
\begin{equation}\label{lem:strong-cocycle-identity}
\beta_w(xy,\xi)
=
\beta_w(y,x^{-1}\xi)+\beta_w(xw,\xi). 
\end{equation}
\label{item:strong-cocycle-identities}
\item for every $w,x\in\Gamma$, the map
\[
\xi\mapsto \beta_w(x,\xi)
\]
is H\"older continuous with the same H\"older constant.
\label{item:strong-cocycle-holder}
\end{enumerate}
\end{lemma}
\begin{proof}
For the first item see the discussion after~\cite[Definition~2.1]{CT24:Invariant}; for the second one, see~\cite[Example~4.12]{CMGR26:GreenMetrics}.
\end{proof}

A \emph{H\"older cocycle} is a map $c \colon \Gamma \times \partial \Gamma \to \mathbb{R}$ satisfying
\[
c(gh,\xi)=c(g,h\xi) + c(h,\xi)
\]
for every $\xi \in \partial\Gamma$, and every $g,h \in \Gamma$, and such that there exists a uniform $\alpha \in (0,1)$ so that for every $g \in \Gamma$, the functions $c(g, \cdot) \colon \partial \Gamma \to \mathbb{R}$ is $\alpha$-H\"older with respect to any visual metric on $\partial \Gamma$.

The Busemann cocycle $\beta$ associated to a strongly hyperbolic metric $d$ induces a H\"older cocycle.
\begin{lemma}
Given a strongly hyperbolic metric $d$ with Busemann cocycle $\beta$, the map
\[
c_d(g,\xi) \coloneqq \beta_0(g^{-1},\xi)
\]
is a H\"older cocycle.
\end{lemma}
\begin{proof}
We have
\[
c_d(g h, \xi) = \beta_o(h^{-1}g^{-1}, \xi) = \beta_o(g^{-1}, h \xi) + \beta_o(h^{-1},\xi) = c_d(g, h \xi) + c_d(h, \xi)
\]
where we used the first item of Lemma~\ref{lem:strong-cocycle-holder}.
Therefore, $c_d$ satisfies the cocycle relation. Finally, the second item of~\ref{lem:strong-cocycle-holder} finishes the proof.
\end{proof}
\section{Equivalent definitions of strong hyperbolicity for currents}
\label{sec:stronghyp}
In this section, we will consider five notions of strong hyperbolicity and show that for pseudometrics dual to geodesic currents, all of these conditions are equivalent. We abbreviate strongly hyperbolic by \emph{s.h.}.

\begin{theorem}
Let $\mu$ be a filling geodesic current on $X$.
The following statements are equivalent:
\[
\begin{array}{ll}
(1)~\emph{$\epsilon$-s.h.\ of $d_\mu^X$}
& (5)~\emph{$(A_0,B_0,C_0)$-s.h.\ of $d_\mu^X$} \\[0.3em]

(2)~\emph{$\epsilon$-s.h.\ with transversals}
& (6)~\emph{$(A_0,B_0,C_0)$-s.h.\ with transversals} \\[0.3em]

(3)~\emph{$\epsilon$-s.h.\ with boxes}
& (7)~\emph{$(A_0,B_0,C_0)$-s.h.\ with boxes} \\[0.3em]

(4)~\emph{$\epsilon$-s.h.\ with crossing pairs}
& (8)~\emph{$(A_0,B_0,C_0)$-s.h.\ with crossing pairs}
\end{array}
\]
\label{thm:stronghyp_equivalent}
\end{theorem}

The conditions in the first line are the two equivalent characterizations of strong hyperbolicity by Nica--\v{S}pakula~\cite{NicaSpakula2016} applied to the pseudometric $d_\mu^X$ (Proposition~\ref{prop:nica_spakula}).

The conditions in the second line are, in some sense, a restatement of the definition of $\epsilon$-strong hyperbolicity and $(A_0,B_0,C_0)$-strong hyperbolicity for  pseudometrics dual to geodesic currents (Lemma~\ref{lem:sh_transversal_vs_distances}, \ref{lem:eps_sh_trans}, \ref{lem:sh_abc_trans}). They feature \emph{(double) transversals}, involving four points in the interior $\wt{X}$. The conditions in the third line use boxes of geodesics, and hence only depend on boundary data (Lemma~\ref{lem:sh_boxes}, \ref{lem:abc_sh_boxes}). Finally, the conditions in the fourth line are expressed in terms of marked length spectrum and show, in particular, that for pseudometrics arising from geodesic currents, strong hyperbolicity is an invariant of the rough isometry class among pseudometrics dual to geodesic currents (Lemma~\ref{lem:stronghyp_abc_cross}, Remark~\ref{rmk:roughisometry}).

We begin by introducing the notion of double transversals. The proof of \cref{thm:stronghyp_equivalent} is then obtained through a sequence of results ranging from \cref{lem:nested_transversals} to \cref{lem:stronghyp_abc_cross}.

\begin{definition}[double transversals]
Given two geodesic segments $I, J$ in $\wt{X}$, let $G(I,J)$ denote the set of geodesics intersecting both $I$ and $J$ transversely. 
Specifically, consider four points $x,y,w,z$ in $\wt{X}$. We define the \emph{double transversal} $G_{x,y,z,w} \coloneqq G([w,x), [y,z))$.
In general, because of atoms, one has to be careful about including interval endpoints. However, we will show that a filling strongly hyperbolic current $\mu$ has no atoms (Corollary~\ref{cor:fillingsh_noatoms}), and hence, in the setting we will be mostly interested about, the specific choice of endpoints will be irrelevant.

In what follows, we use a setup similar to that of~\cite[Lemma~8.15]{MGT25:Intersections}.
Compare Figure~\ref{fig:doubletransversal}.

  \begin{figure}
\centering{
\fontsize{9pt}{8pt}\selectfont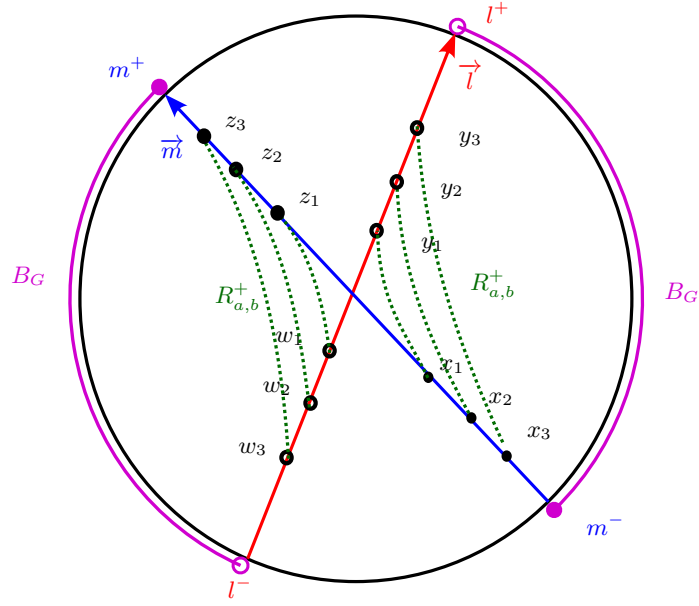}
\caption{Double transversal $G$, nested transversals $G_n$ and box $B_G$.}
\label{fig:doubletransversal}
\end{figure}

For a given choice $G=G_{x,y,z,w}$, we define $G^\perp \coloneqq G([x,y), [z,w))$.
Every transversal $G$ and its opposite transversal $G^{\perp}$ determine two \emph{box extensions}, i.e., two boxes of geodesics $B_G$ and $B_{G^{\perp}}$ obtained by extending the geodesic segments $(w,y)$ (resp. $(x,z)$) to bi-infinite geodesics $l$ (resp. $m$).
We have $G \subset B_G$, $G^{\perp}\subset B_{G^\perp}$, and
\[
(B_G)^{\perp}=B_{G^{\perp}}.
\]
\label{def:box_extension}
\end{definition}

Let $\partial \wt{X}$ denote the Gromov boundary at infinity of $\wt{X}$, which is homeomorphic to $S^1$. Given a double transversal $G$ we can construct two families of nested double transversals $G_n$ (resp. $G_n^{\perp}$) obtained by takings sequences $(y_n), (w_n)$ (resp. $(x_n),(z_n)$) of points $y_n , w_n \in l$ (resp. $x_n, z_n \in m$), so that $w_n \to l^+ \in \partial \wt{X}$ as $n\to \infty$, $y_n \to l^-\in \partial \wt{X}$ as $n\to \infty$ (resp. $z_n \to m^+ \in \partial \wt{X}$ as $n\to \infty$, $x_n \to m^- \in \partial \wt{X}$ as $n\to \infty$), and:
\begin{itemize} \item $w_0\coloneqq w$ and $y_0\coloneqq y$.
\item $w_n < w_{n+1}$ in the natural order in $l$ (resp. $z_n < z_{n+1}$ in the natural order in $m$).
\item $y_n > y_{n+1}$ in the natural order in $l$ (resp. $x_n > x_{n+1}$ in the natural order in $m$).
\end{itemize}
Define $G_n \coloneqq G_{x_n,y_n,z_n,w_n}$ and $G_n^\perp \coloneqq (G_n)^\perp$.
We then have the following.

\begin{lemma} For every geodesic current $\mu$, and for every $n \in \mathbb{N}$, we have
\begin{itemize}
\item $G_n \subset G_{n+1}$, 
\item $B_G=\cup_{n=1}^\infty G_n$ and $G_n^\perp \subset G_{n+1}^\perp$, 
\item $B_G^\perp=\cup_{n=1}^\infty G_n^\perp$.
\item $\lim_n \mu(G_n) = \mu(B_G)$, $\lim_n \mu(G_n^\perp) = \mu(B^\perp_G)$.
\end{itemize}
\label{lem:nested_transversals}
\end{lemma}
\begin{proof}
Items (1) - (3) are a straightforward hyperbolic geometry.
Item (4) follows from the previous items and semi-continuity of measures.
\end{proof}

The following result will be useful in rephrasing strong bolicity and strong hyperbolicity for pseudometrics arising from geodesic currents.

\begin{lemma}
    Let $\mu$ be a geodesic current, and $x,y,z,w \in \wt{X}$ so that $\mu(G[x])=\mu(G[y])=\mu(G[z])=\mu(G[w])=0$.
    Then 
    \[
\mu(G) =\frac{\epsilon}{2}(d_\mu(x,z) + d_\mu(y,w)- d_\mu(x,y) - d_\mu(z,w))
\]
and
\[
\mu(G^\perp) = \frac{\epsilon}{2}(d_\mu(x,z) + d_\mu(y,w)- d_\mu(x,w) - d_\mu(y,z)).
\]
\label{lem:transversal_vs_distances}
\end{lemma}
\begin{proof}
Given $x,y,z,w$ as above, the result follows by~\cite[Lemma~6.2]{DRMG23:Duals} and Definition~\ref{def:current_pseudo}.
\end{proof}

As a consequence, we have the following.

\begin{lemma}\label{lem:sh_transversal_vs_distances}
If $\mu$ is strongly hyperbolic and filling, the equalities in Lemma~\ref{lem:transversal_vs_distances} are satisfied.
\end{lemma}
\begin{proof}
Since $\mu$ is strongly hyperbolic and filling, by Corollary~\ref{cor:fillingsb_noatoms}, $\mu$ has no atoms. Hence, by \cite[Proposition~4.6]{DRMG23:Duals}, $\mu(G[x])=0$ for every $x \in \wt{X}$. Hence, we are under the assumptions of Lemma~\ref{lem:transversal_vs_distances}.
\end{proof}

\begin{lemma}[$\epsilon$-strong hyperbolicity with transversals]
Let $\mu$ be filling.
$d_{\mu}$ is $\epsilon$-strongly hyperbolic if and only if $\mu$ satisfies, for every transversal $G=G_{x,y,z,w}$,
\[
1 \leq e^{-\frac{\epsilon}{2} \mu(G)} + e^{-\frac{\epsilon}{2} \mu(G^{\perp})}
\]
i.e. $\mu$ satisfies \emph{$\epsilon$-strong hyperbolicity with transversals}.
\label{lem:eps_sh_trans}
\end{lemma}
\begin{proof} 
This follows immediately by Lemma~\ref{lem:sh_transversal_vs_distances}. \end{proof}

A similar proof as that of Lemma~\ref{lem:eps_sh_trans} shows the following.

\begin{lemma}[$(A_0,B_0,C_0)$-strong hyperbolicity with transversals]
Let $\mu$ be filling.
Then $d_{\mu}$ is $(A_0,B_0,C_0)$-strongly hyperbolic if and only if $\mu$ satisfies that there exist positive constants $A_0,B_0,C_0$ such that for every $x,y,z,w \in X$ and corresponding transversal $G=G_{x,y,z,w}$,
if
\[
2\mu(G) > A
\]
then
\[
2\mu(G^{\perp})< B_0 e^{-C_0 \cdot A}.
\] In this case, we say $\mu$ satisfies \emph{$(A_0,B_0,C_0)$-strong hyperbolicity with transversals}
\label{lem:sh_abc_trans}
\end{lemma}

Given a box of geodesics $B$, recall that $B^{\perp}$ is the opposite box of geodesics, obtained by considering the complementary intervals determining $B$.

\begin{lemma}[$\epsilon$-strong hyperbolicity with boxes]
A geodesic current is $\epsilon$-strongly hyperbolic if and only, given a positive real number $\epsilon>0$, we have
\[
1 \leq e^{-\frac{\epsilon}{2} \mu(B)} + e^{-\frac{\epsilon}{2} \mu(B^{\perp})}
\]
for every bounded box of geodesics $B$
i.e. $\mu$ satisfies \emph{$\epsilon$-strong hyperbolicity with boxes}
\label{lem:sh_boxes}
\end{lemma}
\begin{proof}
We first prove that strong hyperbolicity implies this inequality. 
Let $G$ denote the transversal determined by three points $x,y,z,w$.
Strong hyperbolicity can be phrased as, for any transversal $G$, having
\[
1 \leq e^{-\frac{\epsilon}{2} \mu(G)} + e^{-\frac{\epsilon}{2} \mu(G^{\perp})}.
\]
Taking a sequence of transversals $G_n$ that approximate the box $B$ we get the inequality.
For the other direction, suppose there exists a transversal $G$ of points $x,y,z,w$ so that
\[
1 > e^{-\frac{\epsilon}{2} \mu(G)} + e^{-\frac{\epsilon}{2} \mu(G^{\perp})}
\]
Consider the box extensions $B=B_G$ and $B^{\perp}=B_G^{\perp}$ of the transversals $G$ and $G^\perp$, respectively (recall Definition~\ref{def:box_extension}).
Note that $G \subset B$ and $G^{\perp} \subset B^{\perp}$. Hence, 
\[
1 > e^{-\frac{\epsilon}{2} \mu(G)} + e^{-\frac{\epsilon}{2} \mu(G^{\perp})} \geq e^{-\frac{\epsilon}{2} \mu(B)} + e^{-\frac{\epsilon}{2} \mu(B^{\perp})} 
\]
contradicting the condition.
\end{proof}

Finally, we have
\begin{lemma}[$(A_0,B_0,C_0)$-strong hyperbolicity with boxes]
Let $\mu$ be filling.
$d_{\mu}$ is $(A_0,B_0,C_0)$-strongly hyperbolic if and only if $\mu$ satisfies that there exist positive real constants $A_0,B_0,C_0$ so 
that if
\[
2\mu(B) > A
\]
then
\[
2\mu(B^{\perp})< B_0 e^{-C_0 \cdot A},
\]
for every box of geodesics $B$.
In this case, we say that $\mu$ satisfies \emph{$(A_0,B_0,C_0)$-strong hyperbolicity with boxes}
\label{lem:abc_sh_boxes}
\end{lemma}
\begin{proof}
Assume $d_{\mu}$ is $(A,B,C)$-strongly hyperbolic. Take a box $B$ and, following Lemma~\ref{lem:nested_transversals}, approximate its endpoints from the interior by four sequences of points $x_n,y_n,z_n,w_n$ in such a way that the corresponding transversals $G_{n}=G_{x_n,y_n,z_n,w_n}$ form a nested sequence and $\cup_n G_n = B$.
Then, by Lemma~\ref{lem:sh_abc_trans}, the result follows.
Suppose on the other hand the assumption is met.
We will show that $\epsilon$-strong hyperbolicity in terms of boxes is met. The proof follows verbatim that of the first implication of \cite[Lemma~6.2]{BS16:StrongHyp}.
We reproduce it here for completeness.
The $\epsilon$-strong hyperbolicity in terms of boxes amounts to showing that, for some $\epsilon>0$, we have,
for every box $B$,
\[
1 \leq \exp(-\epsilon W)+\exp(-\epsilon Z)
\]
where $W \coloneqq \mu(B)$, and $Z \coloneqq \mu(B^{\perp})$.
We may assume that $W,Z\geq 0$, and $W\geq Z$. We rewrite the desired inequality as $\epsilon Z\leq -\log (1-\exp(-\epsilon W))$. As $a\leq -\log(1-a)$ for $a\in [0,1)$, it suffices to obtain $\epsilon Z\leq \exp(-\epsilon W)$. When $W\geq A_0/2$, this can be achieved for $\epsilon \leq \min\{2C_0,2/B_0\}$ by using $(A_0,B_0,C_0)$-strong hyperbolicity in terms of boxes. When $W<A_0/2$, we use the assumption that $d_\mu$ is $\delta$-Gromov hyperbolic: $B_0=\min\{A_0,B_0\}\leq \delta$, where $\delta$ is a hyperbolicity constant of $d_\mu$. This condition follows by approximating every box $B$ by transversals $G_n$, for which the inequality holds. So it suffices to have $\epsilon\delta\leq \exp(-\epsilon A_0/2)$, and this can be achieved for, say, $\epsilon\leq \min\{1/(2\delta), (2\log 2)/A_0\}$.
\end{proof}

Finally, we give a characterization of strong hyperbolicity of $d_\mu$ in terms of its marked length spectrum. 

\begin{lemma}
Let $\mu$ be a filling current on $X$, and $d_\mu$ its dual pseudometric on $\wt{X}$.
Then $d_\mu$ is $\epsilon$-strongly hyperbolic if and only if for every $a,b \in \pi_1(X)$ whose hyperbolic axes $\ora{a},\ora{b}$ cross, we have
\begin{equation}
1 \leq e^{-\epsilon(\ell_{d_\mu}(a)+\ell_{d_\mu}(b)-\ell_{d_\mu}(ab))} + e^{-\epsilon(\ell_{d_\mu}(a)+\ell_{d_\mu}(b)-\ell_{d_\mu}(ab^{-1}))}.
\label{eq:stronghyp_cross}
\end{equation}
In this case, we say that $\mu$ satisfies \emph{$\epsilon$-strong hyperbolicity with crossing pairs}
\label{lem:stronghyp_cross}
\end{lemma}
\begin{proof}
By~\cite[Lemma~8.13]{MGT25:Intersections}, it follows that $\epsilon$-strong hyperbolicity with transversals implies Equation~\ref{eq:stronghyp_cross} for every crossing pair $a,b$.
By~\cite[Lemma~8.14]{MGT25:Intersections}, it follows that Equation~\ref{eq:stronghyp_cross} implies $\epsilon$-strong hyperbolicity with boxes. This proves the equivalence.
\end{proof}

A similar proof shows

\begin{lemma}
Let $\mu$ be a filling current on $X$, and let $d_\mu$ be its dual pseudometric on $\widetilde{X}$.
Then $d_\mu$ is $(A_0,B_0,C_0)$-strongly hyperbolic if and only if for every $a,b \in \pi_1(X)$ whose hyperbolic axes $\ora{a},\ora{b}$ cross, the following holds:

if
\begin{equation}
\ell_{d_\mu}(a) + \ell_{d_\mu}(b) - \ell_{d_\mu}(ab) > A \geq A_0,
\end{equation}
then
\begin{equation}
\ell_{d_\mu}(a) + \ell_{d_\mu}(b) - \ell_{d_\mu}(ab^{-1}) < B_0 e^{-C_0 A}.
\label{eq:stronghyp_abc_cross}
\end{equation}
In this case, we say that $\mu$ satisfies \emph{$(A_0, B_0, C_0)$-strong hyperbolicity with crossing pairs}
\label{lem:stronghyp_abc_cross}
\end{lemma}

\begin{remark}\label{rmk:roughisometry}
In fact, the above two characterizations together with~Theorem~\ref{thm:intersections}, show, in particular, that strong hyperbolicity is a property of the rough isometry class of $d_\mu$ on $\wt{X}$, \emph{provided one restricts to length pseudometrics on $\wt{X}$}.
Indeed, the condition on $\ora{a},\ora{b}$ crossing  is not an invariant across \emph{all} metrics induced from actions of $\Gamma$ on Gromov hyperbolic geodesic spaces $X$.
For example, let $\rho \colon \pi_1(S) \to \PSL(2,\mathbb{C})$ be a \emph{quasiFuchsian} representation, i.e., a representation whose orbit map is a quasi-isometry.
Let $\phi \colon S^1 \to \Lambda_\rho \subset S^2$ be the induced boundary map. Here $\Lambda_\rho$ denotes the limit set of the representation in $S^2$, which is an image of a round circle by a quasi-symmetric map.
If $\rho$ is not Fuchsian, i.e., then by~\cite[Proposition~3.2]{FF22:Quasifuchsian}, the endpoints $\phi(a^+), \phi(a^-)$ and $\phi(b^+),\phi(b^-)$ need not be linked in $\Lambda_\rho$.
\end{remark}
Now we put everything together to prove Theorem~\ref{thm:stronghyp_equivalent}.

\begin{proof}[Proof of Theorem~\ref{thm:stronghyp_equivalent}]
The equivalence between (1) and (5) follows from the work of Nica--\v{S}pakula (Lemma~\ref{prop:nica_spakula}). 
The equivalence (1) and (2) is Lemma~\ref{lem:eps_sh_trans}.
The equivalence (1) and (3) is Lemma~\ref{lem:sh_boxes}.
The equivalence (1) and (4) is Lemma~\ref{lem:stronghyp_cross}.
The equivalence (5) and (6) is Lemma~\ref{lem:sh_abc_trans}.
The equivalence (5) and (7) is Lemma~\ref{lem:abc_sh_boxes}.
The equivalence (5) and (8) is Lemma~\ref{lem:stronghyp_abc_cross}.
\end{proof}

As a corollary of Lemma~\ref{lem:stronghyp_cross}, we have

\begin{corollary}
Let $\mu$ be a filling geodesic current, and $d_\mu^X$ is dual pseudometric for some choice of hyperbolic structure $X$ on $S$. Then
\[
d_\mu^X \text{ is } \epsilon\text{-strongly hyperbolic }
\iff
\forall Y \in \Teich(S), d_\mu^Y \text{ is } \epsilon\text{-strongly hyperbolic}.
\]
\label{cor:stronghyp_roughsimilar}
\end{corollary}

\section{A characterization of strongly bolic currents}
\label{sec:bolic}

We say a geodesic current $\mu$ is \emph{strongly bolic} if its dual pseudometric $d_{\mu}$ is strongly bolic (Definition~\ref{def:strong_bolicity}).
We characterize when a dual space of a current is strongly bolic. For filling geodesic currents, this will be equivalent to the geodesic current having no atoms. We will use this to deduce the main point of this section: strongly hyperbolic filling geodesic currents are non-atomic (Corollary~\ref{cor:fillingsh_noatoms}). The arguments of the main result in this section follow closely~\cite[Proposition~6.9]{DRMG23:Duals}.

We will use the following decomposition theorem for geodesic currents, stated in a form tailored to our purposes.

\begin{theorem}[{\cite[Theorem~1.2(1)]{BIPP21:Currents} (c.f.~\cite[Proposition~A]{EM18:Ergodic})}]
Let $\mu$ be a geodesic current on closed surface $X$. Then there exists a decomposition
\begin{equation}\label{eq:current_decomposition}
\mu \;=\; \nu \;+\; \lambda,
\tag{1.1}
\end{equation}
where, for a finite indexing sets $\mathcal{V}_\mu$, $\mathcal{E}_\mu$, we have:
\begin{itemize}
\item $\nu \coloneqq \sum_{v \in \mathcal V_\mu} \mu_v$, and each $\mu_v$ is a geodesic current whose support in $\mathcal{G}(X)$ projects to a subsurface $\Sigma_v \subset X$, and $\Sigma_v^\circ \cap \Sigma_{v'}^\circ=\emptyset$ whenever $v \neq v'$, where $\Sigma_v^\circ$ denotes the interior of $\Sigma_v$.
\item $\lambda \coloneqq \sum_{c \in \mathcal E_\mu} \lambda_c \, \delta_c$, each $\delta_c$ is the current associated to a simple closed geodesic $c$, and $\lambda_c \geq 0$.
\end{itemize}
In particular, the supports of $\lambda$ and $\nu$ are disjoint.
Moreover, the current $\lambda$ satisfies that if $s$ is a simple closed geodesic in the support of $\mu$ so that $\ora{c} \notin \supp(\nu)$, then $\ora{c} \in \supp(\lambda)$.  We call the set $\mathcal{E}_\mu$ the set of special geodesics, which consists of disjoint simple closed geodesics.
\label{thm:structural_dec}
\end{theorem}

We now prove the following.

\begin{theorem}
Let $\mu$ be a geodesic current decomposing as $\nu + \lambda$ where $\lambda$ is the simple multi-curve corresponding to the special geodesics in the decomposition of $\mu$. Then $\mu$ is strongly bolic if and only if $\nu$ is a non-atomic geodesic current.
\label{thm:bolic}
\end{theorem}
\begin{proof}
Use Theorem~\ref{thm:structural_dec} to write
\[
\mu=\nu+\lambda
\]
where $\lambda$ is the simple multi-curve consisting of the special (simple closed) geodesics of $\mu$, and thus $\nu$ has no special geodesics in its support.

We first prove that if $X_\mu$ is strongly bolic, then $\nu$ is non-atomic.

Suppose, toward a contradiction, that $\nu$ has an atom. Let $\gamma$ be an atomic geodesic in the support of $\nu$, and let $b\in \pi_1(S)$ be the primitive element stabilizing $\gamma$, with endpoints $b^+,b^-\in \partial \widetilde X$. Since $\gamma$ is not special, it intersects some other geodesic in the support of $\nu$. Hence we may choose a box $B$ in the space of geodesics such that
\[
\nu(B)>0
\]
and such that the translates $b^{ki}(B)$, $i\in \mathbb Z$, are pairwise disjoint for some $k>0$.

For each $n\ge 1$, define
\[
B_n:=\bigcup_{i=-n}^n b^{ki}(B).
\]
Since the union is disjoint and $\nu$ is $\pi_1(S)$-invariant,
\[
\nu(B_n)=(2n+1)\nu(B)\xrightarrow[n\to\infty]{}\infty.
\]

Now choose points $x_n,y_n,z_n,w_n\in  \widetilde X$, so that
\[
x_n,y_n\to b^+,
\qquad
z_n,w_n\to b^-,
\]
with each point satisfying  $\mu(G[x_n])=\mu(G[y_n])=\mu(G[z_n])=\mu(G[w_n])=0$, and so that the associated double transversal
\[
G_n:=G_{x_n,y_n,z_n,w_n}
\]
contains $B_n$, while the opposite double transversal $G_n^\perp$ contains $\gamma$. Then
\[
\mu(G_n)\ge \nu(B_n)\xrightarrow[n\to\infty]{}\infty,
\]
whereas
\[
\mu(G_n^\perp)\ge \nu(\{\gamma\})>0.
\]
By Lemma~\ref{lem:transversal_vs_distances}, this contradicts strong bolicity. Therefore $\nu$ must be non-atomic.

We now prove the converse. Assume that
\[
\mu=\nu+\lambda,
\]
where $\nu$ is non-atomic and $\lambda$ is the simple multi-curve of special geodesics. We show that $X_\mu$ is strongly bolic.

Suppose, toward a contradiction, that $X_\mu$ is not strongly bolic. The goal is to show $\nu$ is atomic. Since $X_\mu$ is roughly geodesic, the failure of strong bolicity yields a sequence of quadruples
\[
(x_n,y_n,z_n,w_n)\in (\widetilde X)^4
\]
 so that $\mu(G[x_n])=\mu(G[y_n])=\mu(G[z_n])=\mu(G[w_n])=0$, such that if
\[
G_n:=G_{x_n,y_n,z_n,w_n},
\qquad
G_n^\perp:=G^\perp_{x_n,y_n,z_n,w_n},
\]
then
\[
\mu(G_n)\xrightarrow[n\to\infty]{}\infty
\]
but
\[
\mu(G_n^\perp)\to r
\]
for some finite $r>0$.

Let $B_n:=B_{G_n}$ be the box associated to $G_n$. Then $G_n\subset B_n$, so
\[
\mu(B_n)\xrightarrow[n\to\infty]{}\infty.
\]
On the other hand, $G_n^\perp\subset B_n^\perp$, hence
\[
\mu(B_n^\perp)\ge \mu(G_n^\perp)\ge r/2
\]
for all sufficiently large $n$.

For each $n$, let $m(B_n)\in \widetilde X$ be the center of the box $B_n$, namely the intersection point of the geodesics joining the opposite endpoints of $B_n$. Since the action of $\pi_1(S)$ on $\widetilde X$ is cocompact, after replacing $B_n$ by a suitable translate, we may assume that $m(B_n)$ lies in a fixed compact set $K\subset \widetilde X$ for all $n$. By $\pi_1(S)$-invariance of $\mu$, the above asymptotic properties of $\mu(B_n)$ and $\mu(B_n^\perp)$ are unchanged.

Write
\[
B_n=[a_n,b_n)\times[c_n,d_n),
\]
where $a_n,b_n,c_n,d_n\in \partial \widetilde X$ are in counterclockwise order. Since $\mu(B_n)\to\infty$, the boxes $B_n$ leave every compact subset of the space of geodesics. Because the centers $m(B_n)$ remain in the fixed compact set $K$, it follows that one pair of adjacent endpoints must collapse. Passing to a subsequence, we may therefore assume
\[
d_{\partial \widetilde X}(d_n,a_n)\to 0.
\]
By compactness of $\partial \widetilde X$, after passing to a further subsequence we may assume
\[
a_n,d_n\to x,\qquad b_n\to b,\qquad c_n\to c.
\]

We claim that $x\neq b$ and $x\neq c$. Indeed, if one of these equalities held, then three vertices of $B_n$ would converge to the same point of $\partial \widetilde X$. In that case the center $m(B_n)$ would escape every compact set of $\widetilde X$, contradicting $m(B_n)\in K$. Thus $x,b,c$ are distinct, possibly with $b=c$.

Choose a smaller interval $(b',c')\subset (b,c)$ if $b\neq c$, or a small interval around the common limit point if $b=c$, so that for all sufficiently large $n$,
\[
(b',c')\subset (b_n,c_n).
\]
If $a_n,d_n$ converge to $x$ from different sides, define
\[
B_n'':=(d_n,a_n)\times (b',c').
\]
Otherwise, say that $d_n, a_n$ converge to $x$ in the ccw order. Then, up to subsequence, either $d_n < a_n < x$ or $a_n < d_n < x$. In the first  case, we define 
\[
B_n'':=(d_n,x]\times (b',c'),
\]
and, in the second,
\[
B_n'':=(a_n,x]\times (b',c').
\]
The cw case is handled similarly, where $B_n''$ will have the left endpoints swapped.

Then for all large $n$,
\[
B_n\subset B_n'',
\]
and, after passing to a further subsequence if necessary, we may arrange that
\[
(d_{n+1},a_{n+1})\subset (d_n,a_n),
\]
so that
\[
B_{n+1}''\subset B_n''
\]
for all $n$.

Hence the sets $B_n''$ form a nested sequence of boxes, and their intersection is the pencil
\[
P(x,(b',c')):=\{\text{geodesics with one endpoint }x\text{ and the other in }(b',c')\}.
\]
By continuity from above,
\[
\mu(B_n'')\longrightarrow \mu\bigl(P(x,(b',c'))\bigr).
\]
Since $\mu(B_n)\to\infty$ and $B_n\subset B_n''$, in particular $\mu(B_n'')$ is bounded below by a positive constant for all large $n$. Therefore
\[
\mu\bigl(P(x,(b',c'))\bigr)>0.
\]

By~\cite[Lemma~2.6]{DRMG23:Duals} implies that positive $\mu$-mass on a pencil forces the existence of an atom of $\mu$ with one endpoint equal to $x$. Since $\nu$ is non-atomic, every atom of $\mu$ must come from $\lambda$. Thus $\mu$ has an atomic geodesic $\gamma$ belonging to the support of $\lambda$.
However, since $\gamma \in B_n'' \subset B_n$, it follows that $\gamma$ must intersect transversely the geodesics in $B_n^{\perp}$. But $\mu(B_n^{\perp})>0$ by assumption, so it follows that $\gamma$ the projection of $\gamma$ to $S$ has a self-intersection, which is a contradiction, since the support of $\lambda$ consists of simple closed geodesics.

\end{proof} 

    \begin{figure}[h!]
\centering{
\resizebox{130mm}{!}{\Huge{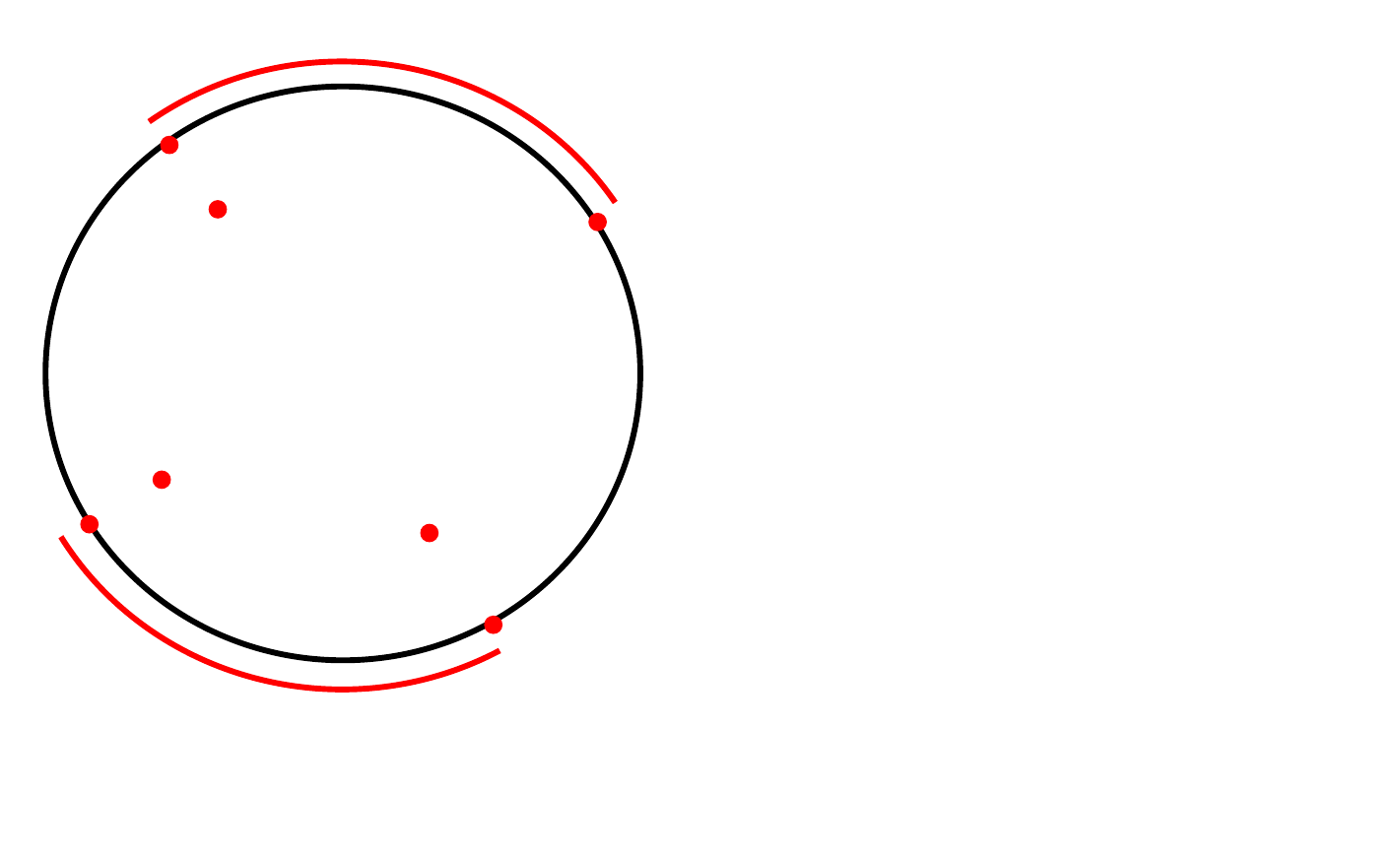}}
    \caption{ }\label{fig:opposite}
}
\end{figure}

\begin{corollary}
    Let $\mu$ be a filling geodesic current. Then $\mu$ is strongly bolic if and only if $\mu$ does not have any atoms.
\label{cor:fillingsb_noatoms}
\end{corollary}
\begin{proof}
   Let $\mu = \nu + \lambda$ be the structural decomposition of the geodesic current $\mu$ as in Theorem~\ref{thm:structural_dec}. Since $\mu$ is filling we have $\mu = \nu$ and $\lambda=0$. Then Theorem~\ref{thm:bolic} implies $\mu$ has no atoms.
\end{proof}
An immediate consequence of this is the following.
\begin{corollary}
A filling strongly hyperbolic current does not have any atoms.
\label{cor:fillingsh_noatoms}
\end{corollary}
\begin{proof}
By~\Cref{item:dmu_roughlygeodesic}, $d_\mu$ is roughly geodesic. By~Lemma~\ref{lem:sh_implies_sb}, $d_\mu$ is strongly bolic, and hence by Corollary~\ref{cor:fillingsb_noatoms}, $\mu$ has no atoms.
\end{proof}

\subsection{Example of strongly bolic and non-strongly hyperbolic current}
\label{subsec:bolic_nonstronglyhyp_CAT}

One might ask whether every strongly bolic current is necessarily strongly hyperbolic. 
The answer is negative: we construct a strongly bolic current which is not $\epsilon$-strongly hyperbolic for any $\epsilon>0$.

\begin{proposition} Let $S$ be endowed with a locally $\CAT(0)$ metric $X$ whose lift to $\widetilde X$ is a $\CAT(0)$ metric containing a flat strip. Then $\wt{X}$ is not $\epsilon$-strongly hyperbolic for any $\epsilon>0$.
\end{proposition}
\begin{proof}
Let the flat strip be represented by
\[
S \cong \mathbb{R}\times [0,x] \subset \mathbb{E}^2
\]
for some $x>0$; see~\cite[Chapter II.2.13]{BH11:NonPosCurvature}. For each $y>0$, consider the Euclidean rectangle
\[
[0,y]\times [0,x]\subset S,
\]
with vertices
\[
p=(0,0), \qquad q=(y,0), \qquad r=(y,x), \qquad s=(0,x).
\]
Suppose, for contradiction, that $\widetilde X$ is $\epsilon$-strongly hyperbolic for some $\epsilon>0$. Applying the $\epsilon$-strongly hyperbolic inequality (Equation~\ref{eq:strong_hyp_epsilon}) to the quadruple $(p,q,r,s)$ gives
\[
e^{\epsilon\sqrt{x^2+y^2}} \leq e^{\epsilon y}+e^{\epsilon x}.
\]
Taking logarithms and using the inequality $\log(1+t)<t$ for every $t>0$, we obtain
\[
\sqrt{x^2+y^2}-y
<
\frac{1}{\epsilon}e^{\epsilon(x-y)}.
\]
On the other hand, since $\sqrt{x^2+y^2}\leq x+y$, it follows that
\[
\sqrt{x^2+y^2}-y
\geq
\frac{x^2}{x+2y}.
\]
Hence
\[
\frac{x^2}{x+2y}
<
\frac{1}{\epsilon}e^{\epsilon(x-y)}.
\]
For fixed $\epsilon,x>0$, the left-hand side decays only on the order of $1/y$, whereas the right-hand side decays exponentially fast as $y\to\infty$. Therefore this inequality fails for all sufficiently large $y$, contradicting $\epsilon$-strong hyperbolicity.

We conclude that $\widetilde X$ is not $\epsilon$-strongly hyperbolic for any $\epsilon>0$.
\end{proof}

Such metrics arise, for example, from singular Euclidean structures associated to holomorphic quadratic differentials with closed trajectories~\cite[Section~5.3]{Hub08:Teich}, such as Jenkins--Strebel differentials determined by simple multicurves.

\section{Negatively curved metrics are not dense}
\label{sec:neg_notdense}
We prove that negatively curved Riemannian metrics are not dense in the space of geodesic currents.

A metric space $(X,d)$ is called \emph{Ptolemy} if for every four points $x,y,z,w \in X$, the following inequality is satisfied
\[
d(x,z)d(y,w) \leq d(x,w)d(y,z) + d(w,z)d(x,y).
\]
Since every four point configuration in a
 $\CAT(0)$-space admits a subembedding into the Euclidean plane (see~\cite[Page 164]{BH11:NonPosCurvature} or \cite[Section~2]{FLS07:Ptolemy}), we have the folowing.
 
\begin{lemma}
Every $\CAT(0)$ metric space is Ptolemy.
\label{lem:cat0_ptolemy}
\end{lemma} 

A length metric $\rho$ on $S$ is called \emph{Ptolemy} if its pullback to $\wt{S}$, the universal cover of $S$, is Ptolemy. 

\begin{lemma}
Let $\rho$ be a locally $\CAT(0)$ length metric on a closed surface $S$. The induced pullback metric on $\wt{S}$ is $\CAT(0)$.
\end{lemma}

Otal proved in \cite{Otal90:SpectreMarqueNegative} that every negatively curved Riemannian metric $\rho$ on $X$ has an associated geodesic current $\mathcal{L}_{\rho}$ with the property that $i(\mathcal{L}_{\rho},g)=\ell_{\rho}(g)$. We will call the geodesic current $\mathcal{L}_{\rho}$ satisfying the above \emph{the dual current to $\rho$}. There are any metrics on $S$ with dual currents.

\begin{enumerate}
\item Non-positively curved Riemannian metrics
\cite[Theorem~A]{CFF92:RigidityNonPosCurvedRiem}.
\item Negatively curved Riemannian metrics with cones of angle $\geq 2\pi$ \cite[Theorem~A]{HP97:RigidityNegCurvedCone}.
\item Singular Euclidean metrics with conical singularities of angle at least $2\pi$ arising from quadratic differentials \cite[Lemma~9]{DLR10:DegenerationFlatMetrics}.
\item More generally, singular Euclidean conical metrics of angle $\geq 2\pi$ \cite[Proposition~3.3]{BL17:RigidityFlat}, \cite{Con18:MarkedNonpos}.
\item Singular Euclidean conical Finsler metrics with conical singularities at least $\geq 2\pi/k$ for $k \in \mathbb{N}$ \cite[Theorem~1.2]{PozzettiShi26:FinslerTranslation}.
\end{enumerate}

 In fact, many more notions of curve length are dual to geodesic currents: from recent work of D. Thurston and the second author~\cite[Corollary~B]{MGT25:Intersections}, we have the following result generalizing all of the above.
  
\begin{theorem}
For every length pseudometric $d_\rho$ on $S$, there exists a geodesic current $\mathcal{L}_{\rho}$ so that
\begin{equation}
\ell_{\rho}(\gamma)=i(\mathcal{L}_\rho,\gamma)
\label{eq:dual_current}
\end{equation}
for every closed curve $\gamma$ in $S$.
\label{thm:intersections}
\end{theorem}

Length pseudometrics on $S$ include any general Riemannian metric, Finsler metric, or any metric conformal to those. These results should convince the reader the assumption of having a dual geodesic current---which will appear in the statement of the next result---is a mild hypothesis.

Given $a, b \in \pi_1(S)$ with corresponding hyperbolic axes $\ora{a},\ora{b}$, we say $a,b$ have \emph{crossing axes} if the endpoints of the axes $\ora{a},\ora{b}$ are cyclically oriented as $(b^-, a^-, b^+, a^+)$.

\begin{theorem}
Let $\rho$ be a Ptolemy metric on $S$ which has a dual geodesic geodesic current $\mathcal{L}_\rho$. For any pair of $a, b \in \pi_1(S)$ with crossing axes, we have, for any $n \in \mathbb{N}$,
\[
\ell_{\rho}(a^n) \ell_{\rho}(b^n) \leq \frac{\ell_{\rho}(a^nb^n)^2}{4} +  \frac{\ell_{\rho}(a^nb^{-n})^2}{4}.
\]
\label{thm:Ptolemy_lengths}
\end{theorem}
\begin{proof}
 Since $\wt{S}$ equipped with $\wt{\rho}$ is Ptolemy, we can apply the Ptolemy relation as follows.

For any $x,y \in \wt{X}$, any $a, b \in \pi_1(X)$, and any $n \in \mathbb{N}$,
apply the Ptolemy relation to the four tuple of elements $y,x, b^n y, a^n x$, obtaining
\begin{equation}
d_{\rho}(x,a^n x)d_{\rho}(y,b^n y) \leq d_{\rho}(y,x) d_{\rho}(a^n x,b^n y) +  d_{\rho}(y,a^n x) d_{\rho}(x,b^n y) 
\label{eq:ptolemy}
\end{equation}
On the other hand, by hypothesis, there exists a geodesic current $\mathcal{L}_\rho$ on $\mathcal{G}(\wt{X})$, for some hyperbolic metric $\wt{X}$, so that
\[
i(\mathcal{L}_\rho, [g]) = \ell_\rho([g])
\]
for every $g \in \pi_1(X)$.
In particular, if $d_{\mathcal{L}_\rho}^{X}$ is the pseudometric induced on $\wt{X}$ by the geodesic current $\mathcal{L}_\rho$ on  $\mathcal{G}(\wt{X})$, we have that
\[
\ell_{d_{\mathcal{L}_\rho}^{X}}([g])=\ell_\rho([g]),
\]
and, hence, by~\cite[Theorem~1.2]{CT25:Manhattan}, $d_{\mathcal{L}_\rho}^{X}$ and $d_\rho$ are roughly isometric, i.e.,
there exist a uniform constant $A>0$ so that for every $p, q \in \wt{X}$,
\begin{equation}
|d_{\mathcal{L}_\rho}^{X}(p,q)-d_\rho(p,q)|<A.
\label{eq:roughly_eq}
\end{equation}
Let $\ora{a}$ and $\ora{b}$ be hyperbolic axes of $a$ and $b$, respectively.
Then, combining Equation~\ref{eq:ptolemy} and Equation~\ref{eq:roughly_eq}, we get the inequality
\begin{equation}
\begin{aligned}
(d_{\mathcal{L}_\rho}^{X}(x,a^n x)-A)(d_{\mathcal{L}_\rho}^{X}(y,b^n y)-A)
&\leq (d_{\mathcal{L}_\rho}^{X}(y,x)+A)( d_{\mathcal{L}_\rho}^{X}(a^n x,b^n y)+A) \\
&\quad + (d_{\mathcal{L}_\rho}^{X}(y,a^n x)+A)(d_{\rho}(x,b^n y)+A).
\end{aligned}
\label{eq:ptolemy2}
\end{equation}

Let $\ora{a},\ora{b}$ denote the hyperbolic axes of $a,b \in \pi_1(X)$, and suppose $\ora{a},\ora{b}$ cross. Assume, furthermore, $x \in \ora{a}$ and $y \in \ora{b}$.

    \begin{figure}[h]
\fontsize{9pt}{8pt}\selectfont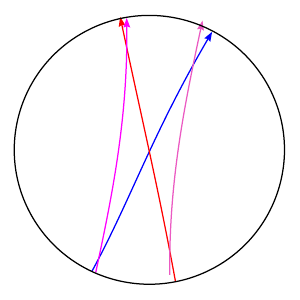
    \caption{ }\label{fig:ptolemy}
\end{figure}

Compare Figure~\ref{fig:ptolemy} for what follows. 
Using~\cite[Lemma~8.12]{MGT25:Intersections}, we get, 
\[
d_{\mathcal{L}_\rho}^{X}(a^n x ,b^n y) = d_{\mathcal{L}_\rho}^{X}( x ,A^n b^n y)
\]
Also, since for any numbers $r,s$, we have $rs \leq (r+s)^2/4$, we can obtain
\begin{equation}
d_{\mathcal{L}_\rho}^{X}(y,x) d_{\mathcal{L}_\rho}^{X}(a^n x,b^n y) \leq (d_{\mathcal{L}_\rho}^{X}(x,y) + d_{\mathcal{L}_\rho}^{X}(a^n x,b^n y))^2/4 = d_{\mathcal{L}_\rho}^{X}(x,a^{-n}b^n x)^2/4
\label{eq:term1}
\end{equation}
where we used that $d_{\mathcal{L}_\rho}^{X}$ is a straight pseudometric (recall Definition~\ref{def:current_pseudo}.
Similarly,
\begin{equation}
d_{\mathcal{L}_\rho}^{X}(y,a^n x) d_{\mathcal{L}_\rho}^{X}(x,b^n y) \leq (d_{\mathcal{L}_\rho}^{X}(y,a^n x)+ d_{\mathcal{L}_\rho}^{X}(a^n x,a^n b^n y) )^2/4 = d_{\mathcal{L}_\rho}^{X}(y,a^{n}b^n y)^2/4
\label{eq:term2}
\end{equation}

Combining Equations~\ref{eq:term1} and~\ref{eq:term2} with Equation~\ref{eq:ptolemy2}, we get the inequality
\begin{equation*}
\begin{aligned}
(d_{\mathcal{L}_\rho}^{X}(x,a^n x)-A)(d_{\mathcal{L}_\rho}^{X}(y,b^n y)-A)
&\leq d_{\mathcal{L}_\rho}^{X}(y,x)\, d_{\mathcal{L}_\rho}^{X}(a^n x,b^n y) \\
&\quad + A d_{\mathcal{L}_\rho}^{X}(a^n x,b^n y)
+ A d_{\mathcal{L}_\rho}^{X}(y,x) \\
&\quad + d_{\mathcal{L}_\rho}^{X}(y,a^n x)\, d_{\rho}(x,b^n y) \\
&\quad + A d_{\rho}(x,b^n y)
+ A d_{\mathcal{L}_\rho}^{X}(y,a^n x) + A^2
\end{aligned}
\label{eq:ptolemy3}
\end{equation*}

Upon dividing both sides of the inequality by $n^2$ and letting $n$ go to infinity, the term
\[
\frac{A d_{\mathcal{L}_\rho}^{X}(a^n x,b^n y)
+ A d_{\mathcal{L}_\rho}^{X}(y,x) + Ad_{\rho}(x,b^n y)+Ad_{\mathcal{L}_\rho}^{X}(y,a^n x) + A^2}{n^2}
\]
on the right hand side goes to $0$, and the desired inequality in terms of stable lengths is obtained using the fact that $\ell_{d_{\mathcal{L}_\rho^X}}([g])=\ell_\rho([g])$.
\end{proof}

By Theorem~\ref{thm:intersections}, every length metric $\rho$ on $S$ determines a dual geodesic current $\mathcal{L}_\rho$.
We denote by $\operatorname{DPtolem}(S)$ the subset of $\C(S)$ consisting of those currents arising from Ptolemaic length metrics, namely
\[
\operatorname{DPtolem}(S)
\coloneqq
\left\{
\mathcal{L}_\rho \in \C(S)
\;\middle|\;
\begin{array}{l}
\text{there exists a Ptolemaic length metric } \\ \rho \text{ on } S 
\text{ such that } \mathcal{L}_\rho \text{ is dual to } \ell_\rho
\end{array}
\right\}.
\]

\begin{lemma}
Let $\gamma$ be a closed curve with at least one self-intersection
and $d_{\gamma}$ the pseudometric on $\wt{X}$ of its associated geodesic current. Then there exists a crossing pair $a,b \in \pi_1(S)$, $n \in \mathbb{N}$ so that
\[
i(\gamma, a^n) i(\gamma,b^n) -\frac{i(\gamma,a^nb^n)^2}{4} -  \frac{i(\mu,a^nb^{-n})^2}{4}>1.
\]
\label{lem:notptolemaic}
\end{lemma}
\begin{proof}
For the following, compare Figure~\ref{fig:filling}.
Since $\gamma$ has a self-intersection, there exists a lift $\wt{\gamma}$ and an element $g \in \pi_1(S)$ so that
$g^k \cdot \wt{\gamma}$ crosses $\wt{\gamma}$ for every $k \in \mathbb{Z}$.
By density of axes of hyperbolic elements, pick $a,b$ crossing and so that $y, x, b^n y, a^n x$ appear in the corners of a bundle of the lifts of $g^{k_n} \wt{\gamma}$ and $\wt{\gamma}$, where here $k_n$ depends on $n$ (in fact, if the translation length of $g$ is very long, compared to $a,b$, one might have to take $n$ larger than 1).
Now, we note:
\[
i(\gamma, a^n) = i(\gamma, b^n) = k_n +1
\]
and
\[
i(\gamma, a^nb^n) = 2 k_n
\]
and
\[
i(\gamma, a^nb^{-n}) = 2
\]
Thus, we have
\[
(k_n+1)^2 - (2k_n)^2/4+ 2^2/4 - k_n^2- 1 = 2k_n>1.
\]
\end{proof}

\begin{corollary}
Let $S$ be a closed surface.  Then $\operatorname{DPtolem}(S)$ is not dense in $\C(S)$.
\label{cor:notdense}
\end{corollary}
\begin{proof}
By Theorem~\ref{thm:intersections}, let $\mathcal{L}_{\rho_n}$ be any sequence of geodesic currents in $\operatorname{DPtolem}(S)$. If $\mu$ is a limit point of such sequence in weak$^*$ topology, then by~Theorem~\ref{thm:Ptolemy_lengths}, and~\cite[Th\'eor\`eme~2]{Otal90:SpectreMarqueNegative}, it must satisfy the condition in Theorem~\ref{thm:Ptolemy_lengths}, i.e., 
\begin{equation}
i(\mu, a^n) i(\mu,b^n) \leq \frac{i(\mu,a^nb^n)^2}{4} +  \frac{i(\mu,a^nb^{-n})^2}{4}.
\label{eq:ptolemaic_length}
\end{equation}
for any pair of $a, b \in \pi_1(S)$ with crossing axes (where $B$ denotes the inverse of $b$), and for any $n \in \mathbb{N}$,

    \begin{figure}[h]
\fontsize{9pt}{8pt}\selectfont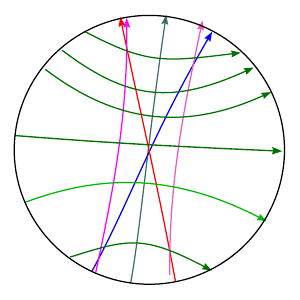
    \caption{ }\label{fig:filling}
\end{figure}

However, by Lemma~\ref{lem:notptolemaic}, the pseudometric associated to a filling curve $\gamma$ is does not satisfy condition~\eqref{eq:ptolemaic_length}.
\end{proof}

We have the following immediate corollary.
Let $\DNonPos(S)$ denote the subspace of geodesic currents dual to locally $\CAT(0)$ length metrics on $S$.

\begin{corollary}
$\DNonPos(S)$ is not dense in the space of geodesic currents.
\label{cor:neg_notdense}
\end{corollary}

$\DNonPos(S)$ includes, for example, geodesic currents dual to negatively curved 
Riemannian metrics.

This contradicts~\cite[Lemma~2.8]{Ham99:Cocycles}, where it is claimed that the Liouville currents dual to negatively curved Riemannian metrics are dense in the space of geodesic currents.
We briefly indicate where the argument breaks down.

On~\cite[Page~108]{Ham99:Cocycles} it is asserted that if geodesic currents $\alpha,\beta$ satisfy
\[
i(\alpha,\gamma)\le i(\beta,\gamma)
\]
for every closed curve $\gamma$, then $\alpha \ll \beta$, i.e.,\ $\alpha$ is absolutely continuous with respect to $\beta$.
However, this implication does not hold in general.

Indeed, let $\beta$ be a figure-eight curve on a genus~2 surface, and let $\alpha$ be the multicurve obtained by resolving its self-intersection into two disjoint simple closed curves.
Then $i(\alpha,\gamma)\le i(\beta,\gamma)$ for every closed curve $\gamma$, while $\alpha$ is not absolutely continuous with respect to $\beta$.

This implication is used in an essential way in the proof of~\cite[Lemma~2.8]{Ham99:Cocycles}, and therefore the argument does not establish the claimed result.
We also note that several subsequent results in~\cite{Ham99:Cocycles}, including~\cite[Theorem~C]{Ham99:Cocycles} (as well as~\cite[Theorems~2.8,~2.9]{Ham99:Cocycles}), rely on this lemma. Nevertheless, we believe those results should hold true by different methods.

\subsection{Examples of strongly hyperbolic geodesic metrics which are not $\CAT(0)$}

Motivated by the preceding discussion, we conclude with a brief comparison of the classes of strongly hyperbolic and Ptolemaic metrics, and show that our methods allow us to construct infinitely many invariant geodesic metrics on $\wt{S}$ which are strongly hyperbolic but not $\CAT(0)$. 

First, we note that it is easy to construct examples of $\CAT(0)$ spaces (and hence Ptolemaic spaces) which are not strongly hyperbolic (Section~\ref{subsec:bolic_nonstronglyhyp_CAT}). On the other hand, any $\CAT(-k)$ metric space with $k>0$ is both strongly hyperbolic~\cite[Theorem~5.1]{NicaSpakula2016} and $\CAT(0)$, hence, Ptolemaic (Lemma~\ref{lem:cat0_ptolemy}). Theorem~\ref{thm:Ptolemy_lengths} allows us to construct infinitely many strongly hyperbolic metrics which are geodesic and not $\CAT(0)$. 

\begin{corollary}
There exists a dense family $\mathcal F$ of geodesic currents such that, for every $\mu\in\mathcal F$, the associated $\pi_1(S)$-invariant pseudometric $d_\mu$ on $\wt S$ is a strongly hyperbolic geodesic metric but not Ptolemaic, and hence not $\CAT(0)$.
\label{cor:sh_not_cat0}
\end{corollary}
\begin{proof}
By Theorem~\ref{mainthm:rteich_dense}, let $\mu_i \in \mathbb{R}^\infty \Teich(S)$ be a sequence of geodesic currents converging to a non-simple curve $\gamma$.
By Lemma~\ref{lem:notptolemaic}, there exist a crossing pair $a,b \in \pi_1(S)$, $n \in \mathbb{N}$ so that for large enough $i$, we have
\[
i(\mu_i, a^n) i(\mu_i,b^n) -\frac{i(\mu_i,a^nb^n)^2}{4} +  \frac{i(\mu_i,a^nb^{-n})^2}{4}>1.
\]
Hence, by Lemma~\ref{thm:Ptolemy_lengths}  $d_{\mu_i}$ is not Ptolemaic.
The $d_{\mu_i}$ are metrics (and not just pseudometrics), since the geodesic currents $\mu_i$ have full support.
The fact that the $d_{\mu_i}$ are geodesic follows from \cite[Proposition~3.6]{DRMG23:Duals}. Finally, since by Lemma~\ref{lem:nonsimple_dense} weighted non-simple closed curves are dense in the space of geodesic currents, the result follows.
\end{proof}

In particular, these metrics have distinct marked length spectra (since they arise from distinct geodesic currents).

The following claim used in the proof above is straighforward, but we include it for completeness.
\begin{lemma}
\label{lem:nonsimple_dense}
Weighted non-simple closed curves are dense in the space of geodesic currents.
\end{lemma}
\begin{proof} By~\cite[Proposition~4.4]{Bonahon86:EndsHyperbolicManifolds}, up to scaling, closed curves are dense in the space of geodesic currents. Hence, it suffices to approximate any simple closed curve $\alpha$ via non-simple closed curves $\beta_n$. This is achieved by taking another simple closed curve $\gamma=[g]$ in $S$ not intersecting $\alpha=[h]$, where $g$ or  $h \in \pi_1(S)$ can be interchanged by $g^{-1}$ or $h^{-1}$ to make $gh^n$ non-simple, and considering the family of weighted non-simple closed curves $\beta_n \coloneqq gh^n/n \to h$ in the weak$^*$-topology.
\end{proof}
Thus, while the classes of Ptolemaic and strongly hyperbolic metrics intersect, neither is contained in the other.

We do not know of other examples of strongly hyperbolic geodesic metrics which are not $\CAT(0)$ on $S$.
The obvious examples of strongly hyperbolic metrics (\cite{NicaSpakula2016}) fail to be $\CAT(0)$ or else are $\CAT(-k)$ for $k>0$:
\begin{itemize}
\item Green metrics, which are typically not geodesic (only roughly geodesic).
\item $\CAT(-k)$ metrics are strongly hyperbolic (by the above discussion).
\end{itemize}

It would be interesting to know the  answer to the following question, which is not precluded by our construction.

\begin{question}
Are the metrics $d_{\mu}$ on $\Gamma=\pi_1(S)$ arising in Corollary~\ref{cor:sh_not_cat0} roughly isometric to any metrics induced by cobounded proper actions of $\Gamma$ on $\CAT(0)$ spaces?
\end{question}

\section{Strongly hyperbolic currents are dense}
\label{sec:stronglyhyp}

In this section we prove that strongly hyperbolic currents are dense in $\mathcal{C}(X)$. To prove this, we will consider a regular finite sheeted covering map $\pi \colon Y \to X$ and an associated transfer map $\Pi: \mathcal{C}(Y) \rightarrow C(X)$. Recall that a covering map is said to be \emph{regular} if the action of its deck group is transitive (equivalently, by the Galois correspondence, the corresponding subgroup is normal).

The transfer map is a natural map to consider for any finite-degree covering of hyperbolic surfaces.

\begin{definition}\label{def3.1}
Suppose $X, Y$ are hyperbolic surfaces.
Let $\pi \colon Y \to X$ be a finite-degree covering map of degree $n$. We define the \emph{transfer map} associated to $\pi$ and denoted by $\Pi: \mathcal{C}(Y) \rightarrow \mathcal{C}(X)$ as 
\[
\Pi(\mu) = \frac{1}{n} \sum_{i=1}^n (g_i)_* \mu
\]
where $\mu \in \mathcal{C}(Y)$ and the set $\{g_1, g_2 \dots, g_n\}$ is a \emph{transversal} of $\pi_1(Y) \leq \pi_1(X)$, i.e., a set of elements in $\pi_1(Y)$ so that $\pi_1(X)$ can be partitioned into left cosets as as $\cup_{i=1}^n g_i \pi_1(X)$. For any box $B$ in $G(\wt{X})$, $(g_i)_*(\mu)(B) = \mu((g_i)^{-1}(B))$. We choose $g_1$ to be the identity element in $\pi_1(X)$. Note that the definition of $\Pi$ depends on the choice of transversal, even though it is not included in the notation.
\end{definition}

Consider a quasiconformal diffeomorphism $f \colon X_1 \to X_2$ between conformally hyperbolic Riemann surfaces,
and lift it to a map $\wt{f} \colon \wt{X_1} \to \wt{X_2}$ between their universal covers.
The Beurling-Ahlfors Theorem~\cite[Corollary~4.9.4]{Hub08:Teich} says that $\wt{f}$ has a continuous extension to the compactification obtained by adding the boundary at infinity
$\wt{f} \colon \wt{X_1} \cup \partial \wt{X_1} \to \wt{X_2} \cup \partial \wt{X_2}$.
We say a box of geodesics $B \subset G(\wt{X_1})$ is \emph{symmetric} if $L_{\wt{X_1}}(B)=\log(2)$.
A homeomorphism $\wt{f} \colon \partial \wt{X_1} \to \partial \wt{X_2}$ is \emph{quasi-symmetric} if the supremum
\[
M(\wt{f}) \coloneqq \sup_{B \text{ symmetric}} \frac{L_{X}(\wt{f}(B))}{\log(2)},
\]
as $B$ ranges over all symmetric boxes $B \subset G(\wt{X_1})$ is finite. 
We let $M(f)$ denote the \emph{quasi-symmetric constant of $f$}. See~\cite[Section 3.3]{BS21:Infinite} for details (in particular, the above definition is equivalent to other usual definitions of quasi-symmetry).

\begin{proposition}[{\cite[Proposition~15]{BS21:Infinite}}]
If a homeomorphism $f \colon \partial \wt{X_1} \to \partial \wt{X_2}$ is quasisymmetric, there exists a homeomorphism $\omega \colon [0, \infty) \to [0, \infty)$ depending only on the quasisymmetric constant of $f$ so that
\[
L_{X_2}(f(B)) \leq \omega(L_{X_1}(B))
\]
for every box of geodesics $B$.
\label{prop:ineq}
\end{proposition}

By the proof of \cite[Proposition~15]{BS21:Infinite}, there exists an increasing homeomorphism $\eta \colon [0,\infty) \to [0, \infty)$, and $\omega(x) \coloneqq \eta^{-1}(M\cdot \eta(x))$, where $M\geq 1$ is a constant depending only on the quasisymmetric constant of $f$.
In fact, using estimates on elliptic integrals, the behaviour of $\omega$ can be estimated explicitly, as we show in Appendix~\ref{sec:elliptic-modulus}. We will use this to prove the following result.

\begin{theorem}
Given a closed hyperbolic surface $X$, any finite positive linear combination of hyperbolic Liouville currents
\[
a_1 \mathcal{L}_{Y_1} + \cdots + a_n \mathcal{L}_{Y_n}
\]
is strongly hyperbolic.
\label{thm:comb_liouville}
\end{theorem}

\begin{proof}
Let $f_i \colon X \to Y_i$ be the marking for $Y_i$, and $\wt{f_i} \colon \partial \wt{X} \to \partial \wt{Y_i}$ its equivariant quasisymmetric lift to the boundary. Recall that we view the Liouville current $L_{Y_i}$ of $Y_i$ as a current $\mathcal{L}_{Y_i}$ on $X$ by setting
\[
\mathcal{L}_{Y_i}(B)\coloneqq (\wt{f_i}^{-1})_*L_{Y_i}(B)=L_{Y_i}(\wt{f_i}(B)).
\]

For $j,k\in\{1,\dots,n\}$, set
\[
\wt{f_{kj}}\coloneqq \wt{f_j}\circ \wt{f_k}^{-1}\colon Y_k\to Y_j.
\]
Since $\wt{f_{kj}}$ is a composition of quasisymmetric homeomorphisms, it is quasisymmetric. Therefore, by Proposition~\ref{prop:ineq}, there exists an increasing homeomorphism $\omega_{kj}\colon [0,\infty)\to[0,\infty)$ such that
\[
\omega_{kj}^{-1}\bigl(\mathcal{L}_{Y_j}(B)\bigr)
\le
\mathcal{L}_{Y_k}(B)
\le
\omega_{kj}\bigl(\mathcal{L}_{Y_j}(B)\bigr)
\]
for every box $B$.

Moreover, by Corollary~\ref{cor:omega-regimes}, for each pair $(k,j)$ there exists a constant $M_{kj}\ge 1$ (depending on the quasisymmetric constant of $f_{kj}$) such that
\[
\omega_{kj}(t)\le M_{kj} t
\qquad\text{for every } t\ge 1.
\]
Let
\[
M\coloneqq \max_{1\le j,k\le n} M_{kj}.
\]
Then
\[
\omega_{kj}(t)\le Mt
\qquad\text{for every } t\ge 1
\]
and every $j,k$.

Now set
\[
\mu\coloneqq \sum_{i=1}^n a_i \mathcal{L}_{Y_i}.
\]
Suppose that
\[
A\le \mu(B)=\sum_{i=1}^n a_i \mathcal{L}_{Y_i}(B).
\]
Then there exists some index $k$ such that
\[
\frac{A}{n a_k}\le \mathcal{L}_{Y_k}(B).
\]
Fix $j\in\{1,\dots,n\}$. Using the comparison inequality above, we get
\[
\frac{A}{n a_k}
\le
\mathcal{L}_{Y_k}(B)
\le
\omega_{kj}\bigl(\mathcal{L}_{Y_j}(B)\bigr),
\]
and therefore
\[
\omega_{kj}^{-1}\!\left(\frac{A}{n a_k}\right)\le \mathcal{L}_{Y_j}(B).
\]

Since $\omega_{kj}(t)\le Mt$ for every $t\ge 1$, it follows that for every $t\ge M$,
\[
\omega_{kj}^{-1}(t)\ge \frac{t}{M}.
\]

Let
\[
a_{\max}\coloneqq \max_{1\le i\le n} a_i.
\]
If
\[
A\ge M n a_{\max},
\]
then we get
\[
\omega_{kj}^{-1}\!\left(\frac{A}{n a_k}\right)\ge \frac{A}{M n a_k}\ge \frac{A}{M n a_{\max}}.
\]
Hence, for every $j$,
\[
\mathcal{L}_{Y_j}(B)\ge \frac{A}{M n a_{\max}}.
\]

Now let $(A_j,B_j,C_j)$ be strong hyperbolicity constants for $\mathcal{L}_{Y_j}$. Choose
\[
A_0\coloneqq \max\!\left\{M n a_{\max},\, M n a_{\max} A_1,\dots, M n a_{\max} A_n\right\}.
\]
Then whenever $A\ge A_0$, we have
\[
\frac{A}{M n a_{\max}}\ge A_j
\qquad\text{for every } j,
\]
and therefore strong hyperbolicity of $\mathcal{L}_{Y_j}$ gives
\[
\mathcal{L}_{Y_j}(B^\perp)
\le
B_j \exp\!\left(- C_j \frac{A}{M n a_{\max}}\right)
\]
for every $j=1,\dots,n$.

Summing over $j$, we obtain
\[
\mu(B^\perp)
=
\sum_{j=1}^n a_j \mathcal{L}_{Y_j}(B^\perp)
\le
\sum_{j=1}^n a_j B_j \exp\!\left(- C_j \frac{A}{M n a_{\max}}\right).
\]
If we now set
\[
B_0\coloneqq \sum_{j=1}^n a_j B_j,
\qquad
C_0\coloneqq \frac{1}{M n a_{\max}} \min_{1\le j\le n} C_j,
\]
then for every $A\ge A_0$,
\[
\mu(B^\perp)\le B_0 e^{-C_0A}.
\]
Thus $\mu$ satisfies the $(A_0, B_0, C_0)$-strong hyperbolicity condition for boxes.
\end{proof}

\begin{proposition}\label{prop2.7}
The transfer map $\Pi \colon \mathcal{C}(Y) \to \mathcal{C}(X)$ has the following properties:
\begin{enumerate}
\item Is continuous when $\mathcal{C}(Y)$ and $\mathcal{C}(X)$ are both equipped with the weak$^*$ topology.
\item Sends non-atomic currents in $\C(Y)$ to non-atomic currents in $\C(X)$.
\item Sends filling currents in $\C(Y)$ to filling currents in $\C(X)$.
\item Sends any hyperbolic Liouville current $\mathcal{L}_{Y}$ in $\C(Y)$ to a strongly hyperbolic current $\Pi(\mathcal{L}_{Y}(B))$.
\item Is 1-Lipschitz when both $\mathcal{C}_{f}(Y)$ and $\mathcal{C}_{f}(X)$ are equipped with their respective $\Delta_Y$ and $\Delta_X$ metrics. In particular, it is (uniformly) continuous when the spaces are equipped with the corresponding metric topologies.
\end{enumerate}
\label{prop:push}
\end{proposition}
\begin{proof}
\begin{enumerate}

\item Follows from the definition.\\

\item For any point $a \in \partial \wt{Y}$, $P(a)$ is a pencil of geodesics in $\wt{Y}$  which has one end point in $a$. A geodesic current $\mu \in \C(Y)$ is non-atomic if and only if $\mu(P(a)) = 0$ for all $a \in \partial \wt{Y}$ ~\cite[Proposition 8.2.8]{Martelli16:IntroGeoTop}. Note that any $g_{i}$ in $\pi_{1}(Y)$ acts on $\wt{Y}$ by isometries and takes pencils to pencils. Now, if $\mu \in \C(Y)$ is a non-atomic current then for any transversal $g_{i}$ we have 
\[
(g_i)_*\mu(P(a))=\mu(g_i^{-1}(P(a))) = \mu(P(g_i^{-1}(a)))= 0
\]
As a result $\Pi(\mu)$ is non-atomic.\\

\item Let $\mu \in \mathcal{C}(Y)$ be a filling current and let $\nu \in \mathcal{C}(X)$. Note that $\mathcal{C}(X) \subset \mathcal{C}(Y)$, since $\pi_1(Y) \leq \pi_1(X)$. Therefore we can think of $\frac{1}{n} \sum_{i=1}^n (g_i)_* \mu$ and $\nu$ as currents in $\mathcal{C}(Y)$. By ~\cite[Lemma~2.22]{Hv24:Commensurating} we get 
\[
    n \cdot i_X\left(\frac{1}{n} \sum_{i=1}^n (g_i)_* \mu, \nu\right) =  i_Y\left(\frac{1}{n} \sum_{i=1}^n {(g_i)_* \mu}, {\nu}\right)
\]
This is because intersection between two geodesics in $X$ lifts to $n$-copies of that intersection in $Y$.
\[
    i_X\left(\frac{1}{n} \sum_{i=1}^n (g_i)_* \mu, \nu\right) = \frac{1}{n^2} \cdot \sum_{i=1}^n  i_Y( {(g_i)_* \mu}, \nu)
\]
 Since $\mu$ is a filling geodesic current in $Y$, the right hand side is non zero.\\

\item Given a Liouville current $\mathcal{L}_Y$, let $\mu_Y \coloneqq \Pi(\mathcal{L}_Y)$.
Note that $d_{\mu_Y}$ is the same regardless of if $\mu_Y$ is considered as a geodesic current in $\C(X)$ or as a current $\wt{\mu_Y}$ in $\C(Y)$.
Note that $\wt{\mu_Y}$ is a linear
combination of Liouille currents, hence, by Theorem~\ref{thm:comb_liouville}, it is strongly hyperbolic.

\item It follows from a straighforward computation:
\begin{align*}
\Delta_X(\Pi(\mu), \Pi(\nu))
&= \log \sup_{\gamma \in \mathcal{C}(X)}
   \frac{i_X\!\left(\sum_i (g_i)_\ast \mu, \gamma\right)}
        {i_X\!\left(\sum_i (g_i)_\ast \nu, \gamma\right)}
 + \log \sup_{\gamma \in \mathcal{C}(X)}
   \frac{i_X\!\left(\sum_i (g_i)_\ast \nu, \gamma\right)}
        {i_X\!\left(\sum_i (g_i)_\ast \mu, \gamma\right)} \\
&= \log \sup_{\gamma \in \mathcal{C}(X)}
   \frac{\sum_i i_Y((g_i)_\ast \mu, \gamma)}
        {\sum_i i_Y((g_i)_\ast \nu, \gamma)}
 + \log \sup_{\gamma \in \mathcal{C}(X)}
   \frac{\sum_i i_Y((g_i)_\ast \nu, \gamma)}
        {\sum_i i_Y((g_i)_\ast \mu, \gamma)} \\
&\leq \log \sup_{\gamma \in \mathcal{C}(X)}
   \max_i \frac{i_Y((g_i)_\ast \mu, \gamma)}
                {i_Y((g_i)_\ast \nu, \gamma)}
 + \log \sup_{\gamma \in \mathcal{C}(X)}
   \max_i \frac{i_Y((g_i)_\ast \nu, \gamma)}
                {i_Y((g_i)_\ast \mu, \gamma)} \\
&= \log \sup_{\gamma \in \mathcal{C}(X)}
   \max_i \frac{i_Y(\mu, (g_i^{-1})_\ast\gamma)}
                {i_Y(\nu, (g_i^{-1})_\ast\gamma)}
 + \log \sup_{\gamma \in \mathcal{C}(X)}
   \max_i \frac{i_Y(\nu, (g_i^{-1})_\ast\gamma)}
                {i_Y(\mu, (g_i^{-1})_\ast\gamma)} \\
&\leq \log \sup_{\eta \in \mathcal{C}(Y)}
   \frac{i_Y(\mu, \eta)}
        {i_Y(\nu, \eta)}
 + \log \sup_{\eta \in \mathcal{C}(Y)}
   \frac{i_Y(\nu, \eta)}
        {i_Y(\mu, \eta)}.
\end{align*}
In the first inequality, we used the mediant inequality,
and in the last inequality we used that we are taking a supremum over a bigger set of closed curves, since every curve in $X$ lifts to a curve in $Y$.

\end{enumerate}
\end{proof}

We have the following classical result of Peter Scott.

\begin{theorem}[{Scott~\cite[Theorem~3.3]{Sco78:SurfaceGeometric} (c.f. \cite{Sco85:SurfaceGeometric_Erratum})}]
Let $X$ be a closed hyperbolic surface and let $\gamma$ be a primitive closed geodesic on $X$.
Then there exists a finite cover $\pi:Y\to X$ such that $\gamma$ admits a simple closed lift to $Y$.
\label{thm:scott}
\end{theorem}

We will need the following immediate variation.

\begin{lemma}
Let $X$ be a closed hyperbolic surface, and let $\gamma$ be a primitive closed geodesic on $X$.
Then there exists a finite regular cover
\[
\pi:Y\to X
\]
such that \emph{every} component of $\pi^{-1}(\gamma)$ is a simple closed geodesic
\label{lem:scott_regular}
\end{lemma}

\begin{proof}
By Theorem~\ref{thm:scott}, there exists a finite cover $p:Y_0\to X$ in which $\gamma$ admits a simple closed lift $\gamma_0\subset Y_0$. Passing to a finite regular cover $\pi:Y\to X$ that factors through $p$, we obtain a covering map $q:Y\to Y_0$.

Let $\wt{\gamma}_0$ be any lift of $\gamma_0$ to $Y$. Since $q$ is a covering map and $\gamma_0$ is simple, the curve $\wt{\gamma}_0$ is simple. As $\pi$ is regular, the deck group acts transitively on the components of $\pi^{-1}(\gamma)$, so every component is a deck translate of $\wt{\gamma}_0$, and hence simple.
\end{proof}

We can finally prove \Cref{mainthm:rteich_dense} and \Cref{mainthm:rinftyeich_sh}. For the definition of $\mathbb{R}\Teich^{\infty}(X)$, refer to Equation~\eqref{eq:rteichinfty}.
\begin{theorem}
The subspace of geodesic currents $\mathbb{R}\Teich^{\infty}(X)$:
\begin{enumerate}
    \item consists of strongly hyperbolic geodesic currents;
    \item is dense in $\C(X)$.
\end{enumerate}
\label{thm:sh_dense}
\end{theorem}
\begin{proof}
First we claim that the subset \[
\{ \lambda \cdot \delta_{\gamma} : \lambda \geq 0, \gamma \text{ primitive closed geodesic}\},\]
as $\lambda$ ranges over all non-negative real numbers and $\gamma$ over all primitive closed geodesics in $X$,
is dense in $\C(X)$.
 This follows immediately from the fact that if in the above set the primitivity condition is lifted, the resulting subset is still dense in the space of geodesic currents. This follows by~\cite[Proposition~4.4]{Bonahon86:EndsHyperbolicManifolds} and since, for any primitive curve $\gamma$ represented by an element $g \in \pi_1(S)$, the non-primitive curve $\gamma^n$ represented by $[g^n]$ corresponds to the geodesic current $n \delta_{\gamma}$, where $\delta_{\gamma}$ is the geodesic current associated to $\gamma$.
 Hence, it suffices to show that any primitive closed geodesic $\gamma$ can be approximated, in the weak$^*$-topology, by a sequence of strongly hyperbolic currents.
If $\gamma$ is a simple closed geodesic, then there exist a sequence of scaled hyperbolic Liouville currents $a_i L_i \to \gamma$~\cite[Proposition~14]{Bonahon88:GeodesicCurrent}, so the result follows.
If $\gamma$ is non-simple, by Lemma~\ref{lem:scott_regular}, there exists a finite regular cover $\pi \colon Y \to X$ such that $\pi^{-1}(\gamma)=\sum_{i=1}^n s_i$, where each $s_i$ is a simple closed geodesic in $Y$. We approximate each $s_i$ by a sequence $a_j^i \mathcal{L}_j^i$ of scaled hyperbolic Liouville currents on $Y$ converging to $s_i$.
 Thus, by Theorem~\ref{thm:comb_liouville}
 $\mathcal{L}_j \coloneqq \sum_{i=1}^n a_i^j \mathcal{L}_i^j$ is strongly hyperbolic, since each $\mathcal{L}_i^j$ is a Liouville current.
Applying the transfer map $\Pi$ associated to the finite cover $\pi$, by Proposition~\ref{prop:push} (4), we get a sequence of strongly hyperbolic currents $\Pi(\mathcal{L}_j)$ converging to $\gamma$ as desired. This proves both items of the statement.
\end{proof}

\section{The correlation counting result}
\label{sec:correlation}

\subsection{Background on thermodynamical formalism and reparameterizations}
We recall the constructions and results that will be needed later. Standard references are~\cite{Bowen1972,Pollicott1987,Sambarino2014,ledrappier1994structure}.

\subsubsection{Translation flows}

Let $\Gamma$ be a non-elementary Gromov hyperbolic group. The translation flow on $\partial^2\Gamma\times\mathbb R$ is defined as
\[
\partial^2\Gamma
=
\{(x,y)\in\partial\Gamma\times\partial\Gamma:x\neq y\},
\qquad
\varphi_t(x,y,s)=(x,y,s-t).
\]
Given a H\"older cocycle $c$ (Section~\ref{subsec:Holderocycle}), the group $\Gamma$ acts on the left as follows
\[
c:\Gamma\times\partial\Gamma\to\mathbb R,
\qquad
\gamma(x,y,t)
=
(\gamma x,\gamma y,t-c(\gamma,y)).
\]
We denote by $M_c$ the quotient of $\partial^2\Gamma\times\mathbb R$ by this action.

\begin{theorem}[Sambarino~\cite{Sambarino2014}]
\label{translationMetricAnosov}
If $c$ has non-negative periods and finite positive entropy, then the action above is proper and cocompact, and the translation flow on $M_c$ is a topologically mixing, metric Anosov flow, admitting a symbolic coding with H\"older roof functions and H\"older conjugate to a H\"older reparametrization of the geodesic flow of a negatively curved surface.
\end{theorem}

In particular, the translation flow admits  Liv\v{s}ic theory, equilibrium states, and the thermodynamic formalism.

\subsubsection{Reparametrization functions}

Given a positive H\"older function $f$, one defines the H\"older reparametrization $\Phi^f$ of a flow $\Phi$ by time change~\cite[\S2]{Sambarino2014}. The periodic orbits are unchanged, and if $\tau$ is periodic then its new period is
\begin{equation}
\label{eqn:NewPeriod}
\lambda(f,\tau)=\int_\tau f.
\end{equation}

The following criterion relates positivity and entropy.

\begin{lemma}[Ledrappier~\cite{ledrappier1994structure}, Sambarino~\cite{Sambarino2014}]
\label{lem:PosEntropy}
Let $f$ be a H\"older function with non-negative periods on a topologically transitive metric Anosov flow. The following are equivalent:
\begin{enumerate}
    \item $f$ is cohomologous to a positive H\"older function;
    \item $f$ has uniformly positive periods;
    \item the entropy $h_f$ is finite and positive.
\end{enumerate}
\end{lemma}

We will also use Ledrappier's correspondence between H\"older cocycles and H\"older functions.

\begin{theorem}[Ledrappier~\cite{ledrappier1994structure}]
\label{thm:Ledrappier}
Let $X$ denote a choice of negatively curved Riemannian metric on $S$.
Every H\"older cocycle $c:\Gamma\times\partial\Gamma\to\mathbb R$
determines a H\"older function $F_c:T^1X\to\mathbb R$
satisfying
$\ell_c(\gamma)=\int_{[\gamma]}F_c$
for all $\gamma\in\Gamma-\{e\}$, and this correspondence induces a bijection on cohomology classes.
\end{theorem}

The following consequence is proved exactly as in~\cite[Proposition~6.11]{DM23:Correlation}.

\begin{lemma}
\label{lem: rep}
Let $c_1,c_2$ be H\"older cocycles. Then there exists a positive H\"older function
$f_{c_1}^{c_2}:M_{c_1}\to\mathbb R_{>0}$
such that for every periodic orbit $\tau$ corresponding to $[\gamma]\in[\Gamma]$,
\[\lambda(f_{c_1}^{c_2},\tau)
=
\ell_{c_2}([\gamma]).\]
\end{lemma}

\subsubsection{Pressure and Liv\v{s}ic cohomology}

For a H\"older function $f:M_c\to\mathbb R$, denote by $P(f)$ its pressure.

\begin{lemma}[Bowen--Ruelle~\cite{BR75:AxiomA}, Sambarino~\cite{Sambarino2014}]
\label{lem:PressueZero}
If $f$ is positive, then
\[
P(-hf)=0
\]
if and only if $h$ is the topological entropy of the reparametrized flow $\Phi^f$.
\end{lemma}

Two H\"older functions are Liv\v{s}ic cohomologous if and only if they have the same periods on all periodic orbits~\cite{Livsic1972}. In particular, pressure depends only on the Liv\v{s}ic cohomology class.
\subsection{The correlation theorem}
\label{sec:corrthm}

In this section we prove the correlation theorem for pairs of strongly hyperbolic currents. The argument follows the strategy of Sharp~\cite{Sharp98:Hyperbolic} and Dai--Martone~\cite{DM23:Correlation}: one realizes the two length spectra as periods of a flow and of a H\"older reparametrization function, and then applies the Lalley--Sharp counting theorem.

\subsubsection{The Lalley--Sharp counting theorem}

Fix a H\"older cocycle $c$, let $\Phi$ denote the associated translation flow on $M_c$, and let
\[
f:M_c\to\mathbb R
\]
be a H\"older continuous function.

The pressure function
\[
t\mapsto P(tf)
\]
is real analytic and strictly convex whenever $f$ is not cohomologous to a constant~\cite[Prop.~4.12]{PP90:Zeta}. Moreover,
\begin{equation}
\label{eq:deriv pressure}
\frac{d}{dt}P(tf)
=
\int f\,d\mu_{tf},
\end{equation}
where $\mu_{tf}$ denotes the equilibrium state of $tf$.

We denote by $J(f)$ the image of the derivative of the pressure function. For each $a\in J(f)$, let $t_a$ be the unique real number satisfying
\[
P'(t_af)=a,
\]
and write
\[
\mu_a=\mu_{t_af}.
\]

Let $c$ be a H\"older cocycle on $\Gamma$.
A flow $\varphi$ on $T^1 M_c$ is said to be \emph{weakly mixing} if its periods
do not generate a discrete subgroup of $\mathbb{R}$, i.e. there exist no non-zero real number $a$ so that
so that $a \ell_\varphi([\gamma]) \in \mathbb{Z}$ for all $[\gamma] \in [\Gamma]$.
Two H\"older cocycles $c_1,c_2$ are \emph{dependent} if
there exist $a_1,a_2 \in \mathbb{R}$, not both equal to zero, so that
\[
a_1 \ell_{c_1}([\gamma]) + a_2 \ell_{c_2}([\gamma]) \in \mathbb{Z},
\]
for all $[\gamma] \in [\Gamma]$. Otherwise, we say $c_1,c_2$ are \emph{independent}.

Two positive H\"older continuous functions
$f,g \colon T^1 M_c \to \mathbb{R}$ are
called \emph{dependent} if there exist $a_1,a_2$ real numbers not both 
equal to zero and a complex valued $C^1$ function
$u \colon T^1 M_c \to S^1$ so that
\[
a_1 f_1 + a_2 f_2 = \frac{1}{2\pi i} \frac{u'}{u}
\]
where $u'$ denotes the derivative of $u$ along the flow, i.e., $\left.\frac{\partial}{\partial t}\right|_{t=0} \bigl(u \circ \varphi_t\bigr)$.
Otherwise, we say $f_1,f_2$ are \emph{independent}.
In particular, $f$ is said to be (in)dependent if $f$ and the constant function $g \equiv 1$ are (in)dependent.

\begin{lemma}
Consider two H\"older cocycles $c_1,c_2$.
If $c_1,c_2$ are independent, then $f_{c_1}^{c_2}$ (as in Lemma~\ref{lem: rep}) is independent.
\label{lem:f_is_indep}
\end{lemma}
\begin{proof}
See~\cite[Remark~2.11]{DM23:Correlation}.
\end{proof}

Until now, we have only relied on general considerations about H\"older continuous cocycles and their associated flow reparameterizations. The following is the only part of this dynamical construction where we use (signed) geodesic currents.

\begin{lemma}[Independence lemma]\label{lem:indep} Consider strongly hyperbolic currents $\mu_1,\mu_2$. If there exist $a,b\in\bb R$ such that $a\ell_{\mu_1}([\gamma])+b\ell_{\mu_2}([\gamma])\in\bb Z$ for all $[\gamma]\in[\Gamma]$, then $a=b=0$.
\end{lemma}

\begin{proof}
By Proposition~\ref{prop:signed_multicurves}, it follows that $a \mu_1+ b \mu_2$ is a signed weighted multi-curve, i.e.
\[
M=\sum_{i=1}^n a_i \gamma_i.
\]
Since $\alpha$ and $\beta$ are not scalar multiples of each other, $M$ is trivial if and only if $a=0$ and $b=0$.
If $M$ is non-trivial, then $a \neq 0$ or $b \neq 0$. But since $M$ is a weighted multi-curve, it is purely atomic and its support is discrete. On the other hand, the support of $a \mu_1 + b \mu_2$, being a linear combination of currents with no atoms (by Theorem~\ref{thm:bolic}), has no atoms, which is a contradiction.
\end{proof}

The following theorem is due independently to Lalley and Sharp.

\begin{theorem}[{Lalley~\cite[Theorem~I]{Lalley97:AxiomA}, Sharp~\cite[Theorem~1]{Sharp92:PrimeOrbit}}]
\label{prop:lalley}
Let $f:M_c\to\mathbb R$ be an independent H\"older continuous function and let $a\in J(f)$. Then, for every fixed $\varepsilon>0$, there exists a constant $C=C(f,\varepsilon)>0$ such that
\[
\#\{
\tau:
\lambda(\tau)\in(x,x+\varepsilon),
\
\lambda(f,\tau)\in(ax,ax+\varepsilon)
\}
\sim
C\frac{e^{h(\mu_a)x}}{x^{3/2}}.
\]
\end{theorem}

Since the translation flows considered below are weakly mixing metric Anosov flows with H\"older symbolic codings, the Lalley--Sharp theorem applies in our setting.

\begin{remark}
\label{rmk:C_secondderivative}
The constant $C=C(f,\varepsilon)$ is comparable to $\varepsilon^2$ and depends on the second derivative of the pressure function. See~\cite[Thm.~5]{Lalley97:AxiomA}, \cite{Sharp98:Hyperbolic}, and~\cite[Theorem~1.2]{BL98:AxiomA}.
\end{remark}

\subsubsection{Pressure intersection}

Let $c_1,c_2$ be H\"older cocycles and let
\[
f=f_{c_1}^{c_2}:M_{c_1}\to\mathbb R_{>0}
\]
be the reparametrization function from Lemma~\ref{lem: rep}.

\begin{definition}
The \emph{pressure intersection} of $c_1$ and $c_2$ is
\[
\mathbf I(c_1,c_2)
=
\int f\,d\mu_{\Phi^{c_1}},
\]
where $\mu_{\Phi^{c_1}}$ denotes the Bowen--Margulis measure of the translation flow on $M_{c_1}$.

The \emph{renormalized pressure intersection} is
\[
\mathbf J(c_1,c_2)
=
\frac{h(c_2)}{h(c_1)}
\mathbf I(c_1,c_2).
\]
\end{definition}

By Liv\v{s}ic theory, these quantities do not depend on the choice of reparametrization function.

\begin{proposition}[{\cite[Proposition~3.8]{BCLS15:Pressure}}]
\label{thm:Jrigidity}
For every pair of H\"older cocycles $c_1,c_2$,
\[
\mathbf J(c_1,c_2)\geq 1,
\]
with equality if and only if
\[
L_{c_1}=L_{c_2},
\]
where $L_c=h(c)\ell_c$ denotes the renormalized length spectrum of the cocycle $c$ (and where $\ell_c=c(g,g^+)$).
\end{proposition}

\subsubsection{Proof of the correlation theorem}

Recall that the renormalized length spectrum of a strongly hyperbolic current $\mu$ is
\[
L_\mu=h(\mu)\ell_\mu.
\]

\begin{theorem}
\label{thm:main}
Fix $\varepsilon>0$. Let $\mu_1,\mu_2$ be strongly hyperbolic currents with distinct renormalized length spectra. Then there exist constants
\[
C=C(\varepsilon,\mu_1,\mu_2)>0
\quad\text{and}\quad
M=M(\mu_1,\mu_2)\in(0,1)
\]
such that
\[
\#\Big\{
[\gamma]\in[\Gamma]:
L_{\mu_1}([\gamma])
\in
(x,x+h(\mu_1)\varepsilon),
\
L_{\mu_2}([\gamma])
\in
(x,x+h(\mu_2)\varepsilon)
\Big\}
\sim
C\frac{e^{Mx}}{x^{3/2}}.
\]
\end{theorem}

\begin{proof}
If is a filling geodesic current $\mu$ whose dual pseudometric $d_\mu$ on $\wt{X}$ is strongly hyperbolic, then it follows from Definition~\ref{def:pseudometric_G}, that the corresponding pseudometric on $\Gamma$ is also strongly hyperbolic.
Let $c_\mu$ be the corresponding H\"older cocycle obtained via Lemma~\ref{lem:strong-cocycle-holder}.
Denote by
$f=f_{\mu_1}^{\mu_2}:M_{\mu_1}\to\mathbb R_{>0}$ the reparametrization function $f_{c_1}^{c_2}$ from Lemma~\ref{lem: rep}, where $c_i \coloneqq c_{\mu_i}$ for $i=1,2$, and $M_{\mu_1}$ denotes the Sambarino translation flow associated to $c_1$, which by Theorem~\ref{translationMetricAnosov} is metric Anosov.
By Lemma~\ref{lem:indep}, the translation flow on $M_{\mu_1}$ is weakly mixing and $f$, by Lemma~\ref{lem:f_is_indep}, is independent. Hence Theorem~\ref{prop:lalley} applies to $f$. The proof now follows verbatim as in the proof of~\cite[Theorem~1.1]{DM23:Correlation}. We briefly sketch it here for completeness. We first show that
$a_0=\frac{h(\mu_1)}{h(\mu_2)}$ belongs to $J(f)$.
By Lemma~\ref{lem:PressueZero},  $P(-h(\mu_2)f)=0.$
Hence
\[
0
=
h(\mu_{-h(\mu_2)f})
-
h(\mu_2)\int f\,d\mu_{-h(\mu_2)f},
\]
so
\begin{equation}
\label{eqn:upper}
\int f\,d\mu_{-h(\mu_2)f}
=
\frac{h(\mu_{-h(\mu_2)f})}{h(\mu_2)}
<
\frac{h(\mu_1)}{h(\mu_2)}.
\end{equation}
The inequality is strict since equality would imply that $h(\mu_2)f$ is cohomologous to a constant, contradicting the independence of $\mu_1$ and $\mu_2$.

On the other hand,
\[
1
\leq
\mathbf J(\mu_1,\mu_2)
=
\frac{h(\mu_2)}{h(\mu_1)}
\int f\,d\mu_\Phi,
\]
where $\mu_\Phi$ denotes the Bowen--Margulis measure of the translation flow. Thus
\begin{equation}
\label{eqn:lower}
\int f\,d\mu_\Phi
>
\frac{h(\mu_1)}{h(\mu_2)}.
\end{equation}

Since $J(f)$ is an open interval and
\[
P'(tf)=\int f\,d\mu_{tf},
\]
Equations~\eqref{eqn:upper} and~\eqref{eqn:lower} imply that there exists
\[
t_{a_0}\in(-h(\mu_2),0)
\]
such that
\[
P'(t_{a_0}f)
=
\frac{h(\mu_1)}{h(\mu_2)}.
\]

Applying Theorem~\ref{prop:lalley} with
\[
a=a_0=\frac{h(\mu_1)}{h(\mu_2)},
\]
we obtain
\[
\#\left\{
[\gamma]\in[\Gamma]:
\ell_{\mu_1}([\gamma])\in(y,y+\varepsilon),
\
\ell_{\mu_2}([\gamma])
\in
\left(
\frac{h(\mu_1)}{h(\mu_2)}y,
\frac{h(\mu_1)}{h(\mu_2)}y+\varepsilon
\right)
\right\}
\sim
\widetilde C
\frac{e^{h(\mu_{a_0})y}}{y^{3/2}}.
\]

Setting \(x=h(\mu_1)y\), we obtain the result with
\[
M=\frac{h(\mu_{a_0})}{h(\mu_1)},
\qquad
C=h(\mu_1)^{3/2}\widetilde C,
\]
and note that
\[
0<M<1,
\]
since \(h(\mu_{a_0})<h(\mu_1)\).
and equality would force $\mu_{a_0}$ to coincide with the Bowen--Margulis measure, contradicting Equation~\eqref{eqn:lower}.
\end{proof}

\subsubsection{Manhattan curves and the correlation number}
\label{subsubsec:manhattan}

Let $c_1,c_2$ be H\"older cocycles. The \emph{Manhattan curve} $\mathcal C(c_1,c_2)$ is the boundary of the set
\[
\left\{
(a,b)\in\mathbb R^2:
\sum_{[\gamma]\in[\Gamma]}
e^{-a\ell_{c_1}([\gamma])-b\ell_{c_2}([\gamma])}
<
\infty
\right\}.
\]

Equivalently,
\[
\mathcal C(c_1,c_2)
=
\{
(a,b)\in\mathbb R^2:
P(-a-bf)=0
\},
\]
where
\[
f=f_{c_1}^{c_2},
\]
and $P$ is the pressure functional.

The following theorem summarizes the properties of the Manhattan curve needed later.

\begin{theorem}
\label{PropertyManhattan}
Let $c_1,c_2$ be H\"older cocycles on $M_c$.
\begin{enumerate}
    \item $\mathcal C(c_1,c_2)$ is a real analytic convex curve passing through $(h(c_1),0)$ and $(0,h(c_2))$.
    
    \item The normals at $(h(c_1),0)$ and $(0,h(c_2))$ have slopes
    \[
    \mathbf I(c_1,c_2)
    \qquad\text{and}\qquad
    \mathbf I(c_2,c_1)^{-1},
    \]
    respectively.
    
    \item $\mathcal C(c_1,c_2)$ is strictly convex unless
    \[
    c_1=c_2.
    \]
\end{enumerate}
\end{theorem}

The proof follows verbatim from~\cite[Theorem~4.1]{DM23:Correlation}; see also~\cite{Bur93:Manhattan,CT25:Manhattan,COR22:Manhattan}.

The correlation number admits the following geometric interpretation.

\begin{theorem}
\label{thm:Manhattan}
Let $\mu_1,\mu_2$ be strongly hyperbolic currents with distinct renormalized length spectra. Then
\[
M(\mu_1,\mu_2)
=
\frac{a}{h(\mu_1)}
+
\frac{b}{h(\mu_2)},
\]
where $(a,b)\in\mathcal C(\mu_1,\mu_2)$ is the unique point whose tangent line is parallel to the line joining $(h(\mu_1),0)$ and $(0,h(\mu_2))$.
\end{theorem}

The proof is identical to~\cite[Theorem~4.2]{DM23:Correlation}.

\subsection{Beyond strongly hyperbolic currents}

Theorem~\ref{thm:main} applies to strongly hyperbolic currents because, in this setting, the corresponding Busemann cocycles are H\"older and the associated translation flows admit the thermodynamic formalism needed for the Lalley--Sharp theorem. Nevertheless, the statement of the counting problem only involves the renormalized length spectra
\[
L_\mu=h(\mu)\ell_\mu.
\]
This suggests asking whether an analogous result should hold for arbitrary filling currents.

There is some evidence for this formulation. The entropy function
\[
h:\mathcal C_{\mathrm{fill}}(S)\to\mathbb R
\]
is continuous, and the Manhattan-curve expression for the correlation exponent extends continuously to pairs of non-proportional projective filling currents. More precisely, since Manhattan curves of points in $D(\Gamma)$ are always $C^1$ (\cite[Theorem~1.1]{CT25:Manhattan}) one can define
\[
\overline M(\alpha,\beta)
=
\frac{a}{h(\alpha)}+\frac{b}{h(\beta)},
\]
where $(a,b)$ is the point on the generalized Manhattan curve whose tangent line is parallel to the line joining
\[
(h(\alpha),0)
\qquad\text{and}\qquad
(0,h(\beta)).
\]
With this definition, and using the strict convexity of $\overline M$, one can check it is continuous on
\[
\mathbb P\mathcal C_{\mathrm{fill}}(S)\times
\mathbb P\mathcal C_{\mathrm{fill}}(S)-\Delta.
\]

This motivates the following question.

\begin{question}
\label{thm:FillingMain}
Fix $\varepsilon>0$. Let $\mu_1,\mu_2$ be two non-proportional filling geodesic currents. Do there exist constants
\[
C=C(\varepsilon,\mu_1,\mu_2)>0
\qquad\text{and}\qquad
M=M(\mu_1,\mu_2)>0
\]
such that, as $x\to\infty$,
\[
\#\Big\{[\gamma]\in[\Gamma]\;\Big|\;
L_{\mu_1}([\gamma])\in \bigl(x,x+h(\mu_1)\varepsilon\bigr),\ 
L_{\mu_2}([\gamma])\in \bigl(x,x+h(\mu_2)\varepsilon\bigr)
\Big\}
\sim
C\frac{e^{Mx}}{x^{3/2}}?
\]
\end{question}

By the above discussion, the main obstruction to extending Theorem~\ref{thm:main} is not the definition of the exponent $M$ (since it extends continuously to all filling geodesic currents), but guaranteeing the convergence of the asymptotic and the existence of the constant $C$. Indeed, $C$ is governed by the second derivative of the pressure, equivalently by the curvature of the Manhattan curve. For general filling currents, it is not clear whether the corresponding Manhattan curve is $C^2$.

\begin{question}
Let $(\alpha,\beta)$ be a pair of filling geodesic currents. Is the Manhattan curve abscissa $\theta_\alpha^\beta(t)$ a $C^2$ function? Is it analytic?
\end{question}

Some evidence supporting this comes from the fact that Manhattan curves are analytic for pairs of distinct strongly hyperbolic currents~\cite[Proposition~2.12]{CT25:Manhattan}, and also for pseudometrics associated to filling closed curves. Indeed, if $\gamma$ is filling, then the pseudometric $d_\gamma$ has the same stable lengths as the geometric action of $\pi_1(S)$ on the dual Sageev complex~\cite{S02:Cube2}, a $\CAT(0)$ cube complex in the scope of~\cite[Theorem~6.1]{CT25:Cubulations}, which implies analyticity of the corresponding Manhattan curve.
\appendix 
\appendix

\section{Signed geodesic currents}
\label{sec:signed}

A \emph{signed geodesic current} on $S$ is a $\Gamma$-invariant Radon signed measure on the space of geodesics $G(\widetilde X)$.
Equivalently,
\[
\mathcal C^{\pm}(S)
=
\{
\mu_1-\mu_2 :
\mu_1,\mu_2\in \mathcal C(S)
\}.
\]

By uniqueness of the Jordan decomposition, if
\[
\alpha=\alpha^+-\alpha^-
\]
is the decomposition of a signed current into mutually singular positive Radon measures, then $\alpha^\pm$ are $\Gamma$-invariant. In particular, $\alpha^\pm \in \mathcal C(S)$.
The associated total variation measure
\[
|\alpha|=\alpha^++\alpha^-
\]
is therefore also a geodesic current.

Bonahon's intersection form extends uniquely by bilinearity to a symmetric bilinear form
\[
i:\mathcal C^\pm(S)\times \mathcal C^\pm(S)\to \mathbb R.
\]

Moreover, if $\gamma$ is a closed curve and $[\gamma)$ is a fundamental domain for an axis of $\gamma$ in $\widetilde X$, then we can extend the intersection number on geodesic currents to signed currents linearly as
\[
i(\alpha,\gamma)=\alpha(G_{[\gamma)}),
\]
where $G_{[\gamma)}$ denotes the set of geodesics intersecting $[\gamma)$.

\subsection{Extension of Otal rigidity}

The following extends Otal's injectivity of marked length spectrum from geodesic currents (\cite[Th\'eor\`eme~2]{Otal90:SpectreMarqueNegative}) to signed currents.

\begin{proposition}\label{prop:signed_otal}
Let $\alpha,\beta\in \mathcal C^\pm(S)$.
If
\[
i(\alpha,\gamma)=i(\beta,\gamma)
\]
for every closed curve $\gamma$, then $\alpha=\beta$.
\end{proposition}

\begin{proof}
Write
\[
\alpha=\mu_1-\mu_2,
\qquad
\beta=\nu_1-\nu_2,
\]
with $\mu_i,\nu_i\in \mathcal C(S)$.

Otal's argument~\cite{Otal90:SpectreMarqueNegative} shows that for every box of geodesics
\[
B=(a,b)\times(c,d)\subset \mathcal{G}(\widetilde X)
\]
(where $a,b,c,d$ are endpoints of lifts of a dense orbit in the unit tangent bundle of $S$),
and for every $\varepsilon>0$, there exist closed curves
\[
\gamma_1,\gamma_2,\gamma_3,\gamma_4
\]
such that for every geodesic current $\mu$,
\[
\left|
\mu(B)
-
\frac{
i(\mu,\gamma_1)
+i(\mu,\gamma_2)
-i(\mu,\gamma_3)
-i(\mu,\gamma_4)
}{2}
\right|
<\varepsilon.
\]
Since the expression above is linear in $\mu$, the same approximation formula holds for signed currents. Therefore,
\[
\left|
\alpha(B)
-
\frac{
i(\alpha,\gamma_1)
+i(\alpha,\gamma_2)
-i(\alpha,\gamma_3)
-i(\alpha,\gamma_4)
}{2}
\right|
<\varepsilon,
\]
and similarly for $\beta$.

Since
\[
i(\alpha,\gamma)=i(\beta,\gamma)
\]
for every closed curve $\gamma$, we obtain
\[
|\alpha(B)-\beta(B)|<2\varepsilon.
\]
Letting $\varepsilon\to0$ gives
\[
\alpha(B)=\beta(B)
\]
for every such box $B$.
Since these boxes have endpoints dense in $S^1$, they generate the topology of $\mathcal{G}(\wt{X})$, and thus the measures $\alpha$ and $\beta$ coincide.
\end{proof}

\subsection{Integral signed currents}

We say that a signed current is a \emph{signed weighted multicurve} if it is a finite linear combination
\[
\alpha=\sum_{i=1}^n a_i\gamma_i,
\]
where $a_i\in\mathbb R$ and $\gamma_i$ are closed curves.

\begin{proposition}\label{prop:signed_multicurves}
Let $\alpha\in \mathcal C^\pm(S)$.
If
\[
i(\alpha,\gamma)\in\mathbb Z
\]
for every closed curve $\gamma$, then $\alpha$ is an signed weighted multicurve.
\end{proposition}

\begin{proof}
Suppose that $\alpha$ is not a signed weighted multicurve.
Then the total variation geodesic current $|\alpha|$ is not purely atomic.
Hence there exists a box $B$ such that
\[
0<|\alpha|(B)<\frac14.
\]
In particular,
\[
|\alpha(B)|<\frac14.
\]

By Proposition~\ref{prop:signed_otal}, for every $\varepsilon>0$ there exist closed curves
\[
\gamma_1,\gamma_2,\gamma_3,\gamma_4
\]
such that
\[
\left|
\alpha(B)
-
\frac{
i(\alpha,\gamma_1)
+i(\alpha,\gamma_2)
-i(\alpha,\gamma_3)
-i(\alpha,\gamma_4)
}{2}
\right|
<\varepsilon.
\]
Since each intersection number is integral,
\[
\frac{
i(\alpha,\gamma_1)
+i(\alpha,\gamma_2)
-i(\alpha,\gamma_3)
-i(\alpha,\gamma_4)
}{2}
\in \frac12\mathbb Z.
\]
Thus $\alpha(B)$ can be approximated arbitrarily well by half-integers.
Since
\[
|\alpha(B)|<\frac14,
\]
the only possible such half-integer is $0$.
Hence
\[
\alpha(B)=0,
\]
contradicting the choice of $B$.
Therefore $|\alpha|$ is purely atomic, and so $\alpha$ is a signed weighted multicurve.
\end{proof}

\begin{remark}
The proof of the proposition shows that $\alpha$ is a half-integral signed multi-curve, although we will not use this. Indeed, the argument implies that $|\alpha|$ and the components of its Hahn decomposition are geodesic currents whose box measures are half-integers. By~\cite[Proposition~4.12]{BIPP21:Crossratios}, twice each of these currents is an integral multi-curve, and hence the currents themselves are half-integral multi-curves. Therefore $\alpha$ is a half-integral signed multi-curve.
\end{remark}
\section{Elliptic integrals and the modulus of a box}
\label{sec:elliptic-modulus}

In this section we express the conformal modulus of a box in terms of complete elliptic integrals and derive the two asymptotic regimes that will be used later.

\subsection{Background on elliptic integrals}

For $k\in(0,1)$, the \emph{complete elliptic integral of the first kind} is defined by
\[
K(k)=\int_0^{\pi/2}\frac{d\theta}{\sqrt{1-k^2\sin^2\theta}},
\]
and we write
\[
k'=\sqrt{1-k^2}
\]
for the \emph{complementary modulus}.

We will use the following two classical facts.

\begin{proposition}[Estimates for $K$]
\label{prop:K-estimates}
There exist absolute constants $c,C>0$ such that:
\begin{enumerate}
\item if $k$ stays in a compact subset of $(0,1)$, then
\[
c\le K(k)\le C;
\]
\item as $k\to 1$, equivalently $k'\to 0$,
\[
c\,\log\frac{4}{k'}\le K(k)\le C\,\log\frac{4}{k'}.
\]
\end{enumerate}
\end{proposition}

\begin{proof}
These estimates are classical; see \cite[§19.12]{DLMF19:Elliptic}. In particular, the second statement follows from the asymptotic expansion
\[
K(k)=\log\frac{4}{k'}+O\bigl((k')^2\log(1/k')\bigr).
\qedhere
\]
\end{proof}

We will also use \emph{Carlson's symmetric elliptic integral}
\[
R_F(x,y,z)=\frac12\int_0^\infty \frac{dt}{\sqrt{(t+x)(t+y)(t+z)}},
\]
which satisfies
\begin{equation}
K(k)=R_F(0,1-k^2,1)
\label{eq:Carlson}
\end{equation}
by \cite[§19.25]{DLMF19:Elliptic}.

Finally, we will use the following rectangular conformal map. If $x_1>x_2>x_3$ are real numbers, then
\[
z(p)=R_F(p-x_1,p-x_2,p-x_3)
\]
maps the upper half-plane conformally onto a rectangle, whose horizontal and vertical side lengths are
\[
W=R_F(0,x_1-x_2,x_1-x_3),
\qquad
H=R_F(0,x_1-x_3,x_2-x_3),
\]
respectively. In particular, the conformal modulus of this rectangle is
\[
\mu=\frac{W}{H}.
\]
See \cite[§19.32]{DLMF19:Elliptic}.

\subsection{Normalization of a box}

Let $Q=[a,b]\times[c,d]$ be a box, where $a,b,c,d\in \partial\widetilde X$ are in positive cyclic order, and let
\[
t=L_{X}(Q)>0
\]
be its Liouville mass. Identifying $\partial \wt{X}$ with $\widehat{\mathbb R}$ and using M\"obius invariance, we may assume that
\[
(a,b,c,d)=(\infty,1,\lambda,0)
\]
for some $\lambda\in(0,1)$.

In this normalization,
\[
L_{X}(Q)=\log \frac{(c-a)(d-b)}{(b-a)(d-c)}=\log\frac{1}{1-\lambda},
\]
and therefore
\[
e^t=\frac{1}{1-\lambda},
\qquad\text{so}\qquad
\lambda=1-e^{-t}.
\]

Thus the conformal modulus of $Q$ depends only on the parameter $t\in(0,\infty)$.

\subsection{Expression of the modulus in terms of $K$}

Apply the rectangular conformal map above with
\[
x_1=1,\qquad x_2=\lambda,\qquad x_3=0.
\]
Then
\[
z(p)=R_F(p-1,p-\lambda,p)
\]
maps the upper half-plane conformally onto a rectangle. Its side lengths are
\[
W=R_F(0,1-\lambda,1),
\qquad
H=R_F(0,1,\lambda)=R_F(0,\lambda,1),
\]
so its conformal modulus is
\[
\mu(Q)=\frac{R_F(0,1-\lambda,1)}{R_F(0,\lambda,1)}.
\]

Now set
\[
k=\sqrt{\lambda},
\qquad
k'=\sqrt{1-\lambda}.
\]
Using Equation~\ref{eq:Carlson}, we obtain
\[
R_F(0,1-\lambda,1)=K(\sqrt{\lambda})=K(k),
\]
and, by symmetry of $R_F$,
\[
R_F(0,\lambda,1)=K(\sqrt{1-\lambda})=K(k').
\]
Hence
\[
\mu(Q)=\frac{K(k)}{K(k')}.
\]

Since $\lambda=1-e^{-t}$, this yields
\[
k=\sqrt{1-e^{-t}},
\qquad
k'=e^{-t/2},
\]
and therefore
\[
\eta(t):=\mu(Q)=\frac{K(\sqrt{1-e^{-t}})}{K(e^{-t/2})}.
\]

\subsection{The two regimes for $\eta$}

We now derive the estimates we will need.
The notation $f(t)\asymp g(t)$ will be used in the relevant
asymptotic regime. More precisely, as $t\to0$, it means that there
exist constants $c,C,t_0>0$ such that
\[
c g(t)\le f(t)\le C g(t)
\qquad\text{for }0<t<t_0,
\]
and as $t\to\infty$, it means that the same estimate holds for
$t>t_0$.

\begin{proposition}
\label{prop:eta-regimes}
The function $\eta:(0,\infty)\to(0,\infty)$ satisfies
\[
\eta(t)\asymp
\begin{cases}
\dfrac{1}{\log(16/t)}, & 0<t\le 1,\\[1.2ex]
t, & t\ge 1.
\end{cases}
\]
\end{proposition}

\begin{proof}
Assume first that $0<t\le 1$. Then
\[
k=\sqrt{1-e^{-t}}\to 0
\qquad\text{as }t\to0.
\]
Since $K$ extends continuously to $0$ and $K(0)=\pi/2$, we have
\[
K(k)\asymp 1.
\]
On the other hand,
\[
k'=e^{-t/2}\to 1
\qquad\text{as }t\to 0,
\]
and the complementary modulus of $k'$ is
\[
\sqrt{1-(k')^2}=\sqrt{1-e^{-t}}=k.
\]
Therefore Proposition~\ref{prop:K-estimates}(2) gives
\[
K(k')\asymp \log\frac{4}{k}.
\]
Since
\[
1-e^{-t}\asymp t
\qquad (0<t\le 1),
\]
we have
\[
k=\sqrt{1-e^{-t}}\asymp \sqrt t.
\]
Hence
\[
\log\frac{4}{k}\asymp \log\frac{16}{t}.
\]
Combining these estimates yields
\[
\eta(t)=\frac{K(k)}{K(k')}
\asymp \frac{1}{\log(16/t)}.
\]

Assume now that $t\ge 1$. Then
\[
0<k'=e^{-t/2}\le e^{-1/2}.
\]
Since $K$ extends continuously to $0$ and is positive on
$[0,e^{-1/2}]$, we have
\[
K(k')\asymp 1.
\]
At the same time,
\[
k=\sqrt{1-e^{-t}}\to 1
\qquad\text{as }t\to\infty,
\]
and its complementary modulus is precisely
\[
\sqrt{1-k^2}=k'=e^{-t/2}.
\]
Thus Proposition~\ref{prop:K-estimates}(2) implies
\[
K(k)\asymp \log\frac{4}{k'}
=\log 4+\frac t2
\asymp t.
\]
Therefore
\[
\eta(t)=\frac{K(k)}{K(k')}
\asymp t.
\]
This proves the result.
\end{proof}

\subsection{The distortion function $\omega$}

The function $\eta$ is strictly increasing and maps $(0,\infty)$
onto $(0,\infty)$. Hence $\eta^{-1}$ is well-defined.
Let $M$ be any real number so that $M\ge 1$ and define
\[
\omega(t)\coloneqq\eta^{-1}(M\,\eta(t)).
\]
In the applications, $M$ will be the quasi-symmetric distortion constant.

\begin{corollary}
\label{cor:omega-regimes}
The function $\omega:(0,\infty)\to(0,\infty)$ satisfies
\[
\omega(t)\asymp
\begin{cases}
t^{1/M}, & 0<t\le 1,\\[1.2ex]
t, & t\ge 1,
\end{cases}
\]
where the implicit constants depend on $M$.
\end{corollary}

\begin{proof}
Let $s=\omega(t)$. By definition,
\[
\eta(s)=M\,\eta(t).
\]

Assume first that $0<t\le 1$. For $t$ sufficiently small, the estimate
\[
\eta(t)\asymp \frac{1}{\log(16/t)}
\]
implies that $\eta(t)$ is small. Since $M$ is fixed, $\eta(s)=M\eta(t)$
is also small, and hence $s$ lies in the small regime as well. Therefore
\[
\frac{1}{\log(16/s)}
\asymp \eta(s)
=M\eta(t)
\asymp \frac{M}{\log(16/t)}.
\]
Thus
\[
\log(16/s)\asymp \frac{1}{M}\log(16/t).
\]
Exponentiating gives
\[
s\asymp t^{1/M}.
\]
After enlarging the implicit constants, the same estimate holds for all
$0<t\le1$. Hence
\[
\omega(t)\asymp t^{1/M}.
\]

Assume next that $t\ge 1$. Then
\[
\eta(t)\asymp t,
\]
so
\[
\eta(s)=M\eta(t)\asymp t.
\]
In particular, for $t$ sufficiently large, $s$ lies in the large regime.
Using $\eta(s)\asymp s$, we obtain
\[
s\asymp t.
\]
Again enlarging the implicit constants if necessary, this holds for all
$t\ge1$. Therefore
\[
\omega(t)\asymp t.
\]
This proves the claim.
\end{proof}

    \bibliography{main}
	\bibliographystyle{hamsalpha}

\end{document}